\documentclass[final,leqno]{siamltex}

\usepackage{amsfonts,amssymb} 
\usepackage{amsmath} 
\usepackage{color} 
\usepackage{graphicx} 
\usepackage{epsfig,mathrsfs} 
\usepackage[bf,SL,BF]{subfigure} 
\usepackage{fancyhdr} 
\usepackage{CJK} 
\usepackage{caption} 
\usepackage{wrapfig} 
\usepackage{cases} 
\usepackage{bm}
\usepackage{algorithm}
\usepackage{algorithmic}
\newtheorem{remark}{Remark}[section] 
\newtheorem{example}{Example}[section] 
\usepackage{amsbsy,amsfonts,amsmath,amssymb}
\usepackage{bbm,calc,caption,color,dsfont,graphicx}
\usepackage{diagbox}
\usepackage{epic,eepic}
\usepackage{multirow}
\usepackage{tikz}
\usepackage{caption}
\usepackage{bm}
\usepackage[bf,SL,BF]{subfigure}
\usepackage{epstopdf}
\usepackage{graphicx,ifthen,latexsym,mathrsfs,multicol}
\usepackage{epstopdf}
\usepackage{bbding}
\usepackage{mathptmx}
\usepackage{algorithm}
\usepackage{algorithmic}
\usepackage{color}
\usepackage{arydshln}
\usepackage{booktabs}
\usepackage{mathtools}

\usepackage{booktabs}
\usepackage{mathtools}

\title{Fast algebraic multigrid    for   block-structured dense and Toeplitz-like-plus-Cross    systems arising from  nonlocal diffusion problems
\thanks{This work was supported by NSFC 11601206. }}

\author{Minghua Chen\thanks{Corresponding author. School of Mathematics and Statistics, Gansu Key Laboratory of Applied Mathematics and Complex Systems,
 Lanzhou University, Lanzhou 730000, P.R. China  (Email: chenmh@lzu.edu.cn)}
\and Rongjun Cao  \thanks{
School of Mathematics and Statistics, Gansu Key Laboratory of Applied Mathematics and Complex Systems,
 Lanzhou University, Lanzhou 730000, P.R. China
(Email: caorj18@lzu.edu.cn) }
 \and Stefano Serra-Capizzano \thanks{
Department of Science and High Technology, University of Insubria, Via Valleggio 11, 22100 Como, Italy $\&$ Department of Information Technology, Division of Scientific Computing, Uppsala University - ITC, L ägerhyddsv. 2, hus 2,  P.O. Box 337, SE-751 05, Uppsala, Sweden
(Email: stefano.serrac@uninsubria.it, stefano.serra@it.uu.se) }
 }

\begin{document}

\maketitle

\begin{abstract}
Algebraic multigrid (AMG) is one of the most efficient iterative methods for solving large  sparse    system of equations.
However,  how to build/check restriction and prolongation operators in practical of AMG methods for nonsymmetric {\em sparse} systems is still
an interesting open question [Brezina, Manteuffel, McCormick, Runge, and Sanders, SIAM J. Sci. Comput. (2010); Manteuffel and Southworth, SIAM J. Sci. Comput. (2019)].
This paper deals with the block-structured dense and  Toeplitz-like-plus-Cross systems, including {\em nonsymmetric} indefinite, symmetric positive definite (SPD), arising from  nonlocal diffusion  problem and peridynamic problem.
The simple (traditional)  restriction operator and prolongation operator are employed in order to handle such  block-structured dense and  Toeplitz-like-plus-Cross  systems, which  is convenient and efficient when employing a fast AMG.
We focus our efforts on  providing  the detailed proof of the convergence of the two-grid method for  such SPD situations.
The numerical experiments are performed in order to verify the convergence with a  computational cost of only  $\mathcal{O}(N \mbox{log} N)$ arithmetic operations, by using few fast Fourier transforms,  where $N$ is the number of the grid points.
To the best of our knowledge, this is the first contribution regarding Toeplitz-like-plus-Cross linear systems solved by means of a fast AMG.
\end{abstract}



\pagestyle{myheadings}
\thispagestyle{plain}
\markboth{M. H. CHEN, R.  CAO, AND S. SERRA-CAPIZZANO}{FAST AMG FOR  BLOCK-STRUCTURED DENSE     SYSTEMS}

\section{Introduction}\label{sec:1}
Large, sparse, block-structured  linear systems arise in a wide variety of applications throughout computational science and engineering
involving   advection-diffusion flow \cite{Sivas2021siam}, image process \cite{NgPanWeighted2014}, Markov chains \cite{Stewart:94},  Biot's consolidation model \cite{QSW:21},
Navier-Stocks equations and   saddle point  problems \cite{BenziGolub2005}.
In this paper we study the fast algebraic multigrid  for solving the  block-structured dense  linear systems, stemming  from nonlocal problems \cite{Andreu:10,CCNWL:20,Chen:1--30,Du:19} by the piecewise quadratic polynomial collocation approximations, whose associated matrix can be expressed as a $2\times 2$ block structure
\begin{equation}\label{eq: blocksystem}
\mathcal{A}u=\left(\begin{array}{cc} A & B\\ C & D \end{array}\right)\left(\begin{array}{c} v \\ w  \end{array}\right) =\left(\begin{array}{c} b_f \\ b_g  \end{array}\right),
\end{equation}
with   the coefficient matrices $A\in \mathbb{R}^{M\times M}$,  $B\in \mathbb{R}^{M\times N}$,  $C\in \mathbb{R}^{N\times M}$ and $D\in \mathbb{R}^{N\times N}$.

Algebraic multigrid (AMG) is one of the most efficient iterative methods for solving large-scale     system of equations \cite{RugeSIAM1987,Xu:17}.
In the past decades, AMG methods for linear systems having  Toeplitz coefficient matrices with scalar entries  have been widely studied \cite{CCS:98}
including  elliptic PDEs  \cite{RugeSIAM1987,Serra:02,XuMMEPDES1996,Xu:17}, fractional PDEs \cite{CWCD:14,DMS:18,Pang:12} and nonlocal PDEs \cite{CES:20,ChenDeng2017}.
Some papers have investigated the case of block entries, where the entries are small generic matrices of fixed size instead of scalars \cite{ChenDeng2018V-cycle,WILEY2020BLOCKKt,Fiorentino:96}.
Only few papers have investigated block-structured {\em sparse} linear systems \eqref{eq: blocksystem}.
For example, setting up partial prolongations operators with a Galerkin coarse grid matrix, a new AMG approach for Stokes problem are designed in \cite{NotayStokes2016}.
Constructing  the corresponding tentative transfer operators, a fully aggregation-based AMG is developed for nonlinear contact problems of saddle point \cite{WMPGW:21}.
Using approximate ideal restriction (AIR) operators, AIR AMG method for space-time hybridizable discontinuous Galerkin discretization of advection-dominated flows are investigates \cite{Sivas2021siam}.
Defining the interpolation matrix with the coarse coefficient vector,
multilevel Markov chain Monte Carlo AMG algorithms are performed.
Involving sparse integral transfer operators, towards adaptive smoothed aggregation AMG for nonsymmetric problems are considered \cite{Bre2010aggregation}.
A transfer operator based on the  fractional approximation property, two-grid methods convergence in norm of nonsymmetric algebraic multigrid are presented \cite{Bre2010aggregation}.
However,  how to build/check restriction and prolongation operators in practical of AMG methods for nonsymmetric {\em sparse} systems is still
an interesting open question \cite{Bre2010aggregation,MaNs:19}.
In particular, how to develop/design  fast AMG for   block-structured dense  linear systems \eqref{eq: blocksystem} is still  an interesting problems, since the above  special prolongation/transfer operators are not easy to be employed in connection with the fast Fourier transform.

In this work,  the simple (traditional)  restriction operator and prolongation operator are used in order to handle such  block-structured dense   systems  \eqref{eq: blocksystem}, including  nonsymmetric indefinite system, symmetric positive definite (SPD), Toeplitz-plus-diagonal systems, which derive  from the  nonlocal problems discussed in \cite{Andreu:10,CCNWL:20,Chen:1--30,Du:19}.
In general, it is still not at all easy for dense stiffness matrices  \cite{Arico:07,Arico:04,Bolten:15,ChenDeng2017}, unless we can reduce the problem to the Toeplitz setting and we know the symbol, its zeros, and their orders  \cite{Serra:02}. Instead we will focus our attention on first answering such a question for a two-level setting, since  it is useful from a theoretical point of view as first step: in fact the study of the MGM convergence usually begins from the convergence analysis of the two-grid method (TGM) \cite{Pang:12,RugeSIAM1987,Xu:17}.
We focus our attention in providing a detailed proof of the convergence of the two-grid method (TGM) for the considered SPD linear systems.
To the best of our knowledge, this is the first time that a fast AMG is studied for  the block-structured dense linear systems as those reported in \eqref{eq: blocksystem}.

The outline of this paper is as follows.
In the next section, we introduce block-structured dense  systems including  applications in  nonlocal diffusion  problems and peridynamic problem by the piecewise quadratic polynomial collocation.
In Section 3, block-structured V-cycle AMG algorithm by Fast fourier transform for Toeplitz-like-plus-Cross systems are designed.
Convergence rate of the two-grid method is analyzed in Section 4. To show the effectiveness of the presented schemes, results of numerical experiments are reported in section 5.
Finally, we conclude the paper with some remarks and open problems.

\section{Block-structured dense  systems applications}

Nonlocal diffusion problems have been used to model very different scientific phenomena occurring in various applied fields, for example in biology, particle systems, image processing, coagulation models, mathematical finance, etc. \cite{Andreu:10,Du:19}.
Recently, the nonlocal volume-constrained diffusion problems,  the  so-called nonlocal model for distinguishing the nonlocal diffusion problems, attracted the wide interest of scientists \cite{Du:19}, where the linear scalar peridynamic model can be considered as a special case  \cite{Du:19,Silling:00}.
 For example, the nonlocal peridynamic (PD) model  is becoming an attractive emerging tool for the multiscale material simulations of crack nucleation and growth, fracture, and failure of composites \cite{Silling:00}.

Let  the  piecewise quadratic base functions be defined by  \cite[p.\,37]{Aikinson:09}
\begin{equation*}
\phi_{m}(x)=\left\{\begin{array}{rl}  \frac{x-x_{m-1}}{h}\frac{2x-(x_{m-1}+x_m)}{h}, & x\in\left[x_{m-1},x_m\right],\\
 \frac{x_{m+1}-x}{h}\frac{(x_{m+1}+x_m)-2x}{h}, & x\in\left[x_m,x_{m+1}\right],\\
 0,& {\rm otherwise};
  \end{array}\right.
\end{equation*}
and
\begin{equation*}
\phi_{m-\frac{1}{2}}(x)=\left\{\begin{array}{rl}  \frac{4(x-x_{m-1})(x_m-x)}{h^2}, & x\in\left[x_{m-1},x_m\right],\\
 0, & {\rm otherwise}.
  \end{array}\right.
\end{equation*}
Then the piecewise Lagrange quadratic interpolant of $u(x)$ is
\begin{equation}\label{{ccs2.1}}
u_{Q}(x) =\sum_{j=-r}^{ N+r}u(x_j)\phi_j(x)+\sum_{j=-r}^{ N+r-1}u(x_{j+\frac{1}{2}})\phi_{j+\frac{1}{2}}(x).
\end{equation}
In this work,  we mainly focus on two types of nonlocal problems approximated  by the piecewise quadratic polynomial collocation,
which give raise to the  block-structured dense  systems expressed in \eqref{eq: blocksystem}.
\subsection{Application in  nonlocal  model }
Consider the   time-dependent  nonlocal  diffusion problem, whose  prototype
is  \cite{Andreu:10,Bates:06,CCNWL:20,Chen:1--30}
\begin{equation}\label{modelnonlocal}
\left\{ \begin{split}
 u_t - \mathcal{L} u   &=f  &  ~~{\rm on}  & ~~\Omega,\, t>0,\\
                        u(x,0)&=u_0      &  ~~{\rm on}  & ~~\Omega,
 \end{split}
 \right.
\end{equation}
with   Dirichlet  boundary conditions.
The  nonlocal operator $\mathcal{L}$  is defined by
\begin{equation*}
\begin{split}
\mathcal{L}u(x)=\int_\Omega J(|x-y|)\left[ u(x)-u(y) \right]dy ~~\forall x \in \Omega.
\end{split}
\end{equation*}
Here $ J(x)\sim\frac{1}{x^{1+2s}},~~s\in [-0.5,1)$ is a radial probability density with a nonnegative  symmetric dispersal kernel.

We  briefly review some basic relevant notions concerning the piecewise quadratic polynomial collocation approximations
for the corresponding stationary problem
\begin{equation*}
 \int^b_a \frac{u(x)-u(y)}{|x-y|^\gamma}dy  =f(x),   \quad 0< \gamma <1,
\end{equation*}
which leads to the nonsymmetric and indefinite system  \cite{CCNWL:20,Chen:1--30}
\begin{equation}\label{3system}
\begin{split}
\mathcal{A}_hU_h=\eta_{h,\gamma}\cdot (F_h+K_h)~~ {\rm with}~~
\mathcal{A}_h=\left ( \begin{matrix}
 \mathcal{D}_{1}  & 0\\
 0      &\mathcal{D}_{2}
 \end{matrix}
 \right )
-
\left ( \begin{matrix}
  \mathcal{M}     & \mathcal{Q }  \\
 \mathcal{ P }        & \mathcal{N}
 \end{matrix}
 \right )
\end{split}
\end{equation}
and $\eta_{h,\gamma}=\frac{(3-\gamma)(2-\gamma)(1-\gamma)}{h^{1-\gamma}}$. Here the diagonal  matrices $\mathcal{D}_1$ and  $\mathcal{D}_2$  are  given by
$$\mathcal{D}_1={\rm diag}\left(d_1,d_2,\ldots, d_{N-1}\right),~ \mathcal{D}_2={\rm diag}\left(d_\frac{1}{2},d_\frac{3}{2},\ldots, d_{N-\frac{1}{2}}\right)$$
with
\begin{equation*}
\begin{split}
 d_{\frac{i}{2}} = (3-\gamma)(2-\gamma)\left( \left(\frac{i}{2}\right)^{1-\gamma} + \left(N-\frac{i}{2}\right)^{1-\gamma} \right), \quad i=1,2,\cdots,2N-1.
\end{split}
\end{equation*}
The square matrices $\mathcal{M}$, $\mathcal{N}$ and {\em rectangular matrices} $\mathcal{P}$, $\mathcal{Q}$  are, respectively, defined by
\begin{equation*}
\begin{split}
\mathcal{M}&={\rm toeplitz}\left(m_0,m_1,\ldots, m_{N-2}\right),~~\mathcal{N}={\rm toeplitz}\left(n_0,n_1,\ldots, n_{N-1}\right),\\
\mathcal{P}&={\rm toeplitz}([p_{0},p_{0},p_{1},\dots,p_{N-2}],[p_{0},p_{1},\dots,p_{N-2}]),\\
\mathcal{Q}&={\rm toeplitz}([q_{0},q_{1},\dots,q_{N-2}],[q_{0},q_{0},q_{1},\dots,q_{N-2}]).
\end{split}
\end{equation*}
The corresponding  coefficients are computed by $m_0=2(1+\gamma)$, $n_{0}={ (2-\gamma) 2^{\gamma+1}}$ and
\begin{equation*}
\begin{split}
 p_0&=4\left[\left(\frac{3}{2}\right)^{3-\gamma}-\left(\frac{1}{2}\right)^{3-\gamma}\right]-(3-\gamma)\left[\left(\frac{3}{2}\right)^{2-\gamma}+3\left(\frac{1}{2}\right)^{2-\gamma}\right],\\
 q_{k}&=-8\left((k+1)^{3-\gamma}-k^{3-\gamma}\right)+4(3-\gamma)\left((k+1)^{2-\gamma}+k^{2-\gamma}\right),\, k \ge 0,\\
  m_k&=4\left[ (k+1)^{3-\gamma}\!-\!(k-1)^{3-\gamma} \right]
-(3-\gamma)\left[ (k+1)^{2-\gamma}+6k^{2-\gamma}+(k-1)^{2-\gamma}  \right],~k\geq 1
\end{split}
\end{equation*}
with $p_{k}=m_{k+\frac{1}{2}}$, $n_{k}=q_{k-\frac{1}{2}}$ for $k\geq 1$.

The boundary data $K$ is given by
\begin{equation*}
\begin{split}
  K&=\left(\eta_{1},\eta_{2},\cdots,\eta_{N-1},\eta_\frac{1}{2},\eta_\frac{3}{2},\cdots,\eta_{N-\frac{1}{2}}\right)^{T}u_0\\
&\quad+\left(\eta_{N-1},\eta_{N-2},\cdots,\eta_{1},\eta_{N-\frac{1}{2}},\eta_{N-\frac{3}{2}},\cdots,\eta_{\frac{1}{2}}\right)^{T}u_N
\end{split}
\end{equation*}
with $\eta_{\frac{1}{2}}={{(2-\gamma)(1-\gamma)} 2^{\gamma-1}}$ and
$$\eta_{i/2}
= 4\left[ i^{3-\gamma} \!- \!(i-1)^{3-\gamma}  \right]
-(3-\gamma) \left[ 3i^{2-\gamma} + \left(i-1\right)^{2-\gamma} - (2-\gamma)i^{1-\gamma}\right],~~2\leq i\leq 2N-1.$$

\subsection{Application in  Peridynamic  model}
Let us consider the following  time-dependent  peridynamic/nonlocal volume-constrained diffusion problem \cite{ChenDeng2017,Du:19,Silling:00}
\begin{equation} \label{ccs2.4}
\left\{ \begin{split}
 u_t - \mathcal{L}_\delta u   &=f  &  ~~{\rm on}  & ~~\Omega,\, t>0,\\
                        u(x,0)&=u_0      &  ~~{\rm on}  & ~~\Omega \cup \Omega_\mathcal{I},\\
                 u&=g       &  ~~{\rm on}  & ~~ \Omega_\mathcal{I},t>0.
 \end{split}
 \right.
\end{equation}
The  nonlocal operator $\mathcal{L}_\delta$  is defined by \cite{Du:19}
\begin{equation*}
\begin{split}
\mathcal{L}_\delta u(x)=\int_{B_\delta(x)} \gamma_\delta(|x-y|) \left[u(y)-u(x)\right]dy\ \  ~~\forall x \in \Omega
\end{split}
\end{equation*}
with $B_\delta(x)=\{y \in \mathbb{R}: |y-x|<\delta \}$ denoting a neighborhood centered at $x$ of radius $\delta$,
 which is the horizon parameter and  represents the size of nonlocality;
 the symmetric nonlocal kernel is defined as $\gamma_\delta(|x-y|)=0$  if $y \not \in B_\delta(x)$.

Before we start to discuss this problem we shall briefly review few preliminary notions regarding the piecewise quadratic polynomial collocation approximations for the corresponding stationary problem
\begin{equation} \label{ccs2.5}
\left\{ \begin{split}
  - \mathcal{L}_\delta u      &=f ~~{\rm on}   ~~\Omega,\\
                 u&=g       ~~\,  {\rm on}   ~~ \Omega_\mathcal{I},
 \end{split}
 \right.
\end{equation}
where $u=g $ denotes a  volumetric constraint imposed on a volume $\Omega_\mathcal{I}$ that has
a nonzero volume and is made to be disjoint
 from $\Omega$. In order to keep the expression simple, below we assume the unit interval $\Omega$ with the volumetric constraint domain $\Omega_\mathcal{I}=[-\delta,0]\cup[1,1+\delta]$,
but everything can be shifted to an arbitrary interval $[a,b]$.
For convenience,  we focus on the special case where the kernel $\gamma_\delta(s)$ is taken to be a constant, i.e., $\gamma_\delta(s)=3\delta^{-3}$ \cite{ChenDeng2017,Du:19,Tian:13}. More general kernel types \cite{Du:19,Tian:13}  can be studied in a similar manner.

Now, we introduce and discuss the discretization scheme of (\ref{ccs2.5}). Let
the ratio $r=\big\lfloor\delta/h\big\rfloor\ge1$ if $\delta\geq  h$ and $r=\big\lceil\delta/h\big\rceil=1$
if $\delta\le h$.
Let the mesh points $x_i=ih$, $h=1/N$,  $i \in \mathcal{N}$ and
\begin{equation*}
\begin{split}
  &\mathcal{N}=\left\{-r,{-r+\frac{1}{2}},\ldots,0,\frac{1}{2},\ldots,{N-\frac{1}{2}},{N},\ldots,{N+r-\frac{1}{2}},N+r\right\},\\
  &\mathcal{N}_{in}=\left\{{\frac{1}{2}},1,\frac{3}{2}, \ldots, N-\frac{1}{2}\right\},~~\mathcal{N}_{out}=\mathcal{N}\setminus\mathcal{N}_{in}.
\end{split}
\end{equation*}
Let $u_i$ be the approximated
value of $u(x_i)$ and $f_{i} = f(x_i)$.
Denote $I_m=\left[(m-1)h,mh\right]$, $1\leq m\leq r$, and $I_{r+1}=[rh,\delta]$.

\subsubsection{ Nonsymmetric indefinite   block-structured dense  systems}

 By the piecewise quadratic polynomial collocation \eqref{{ccs2.1}},
it is easy to check that the standard collocation method  of stationary problem \eqref{ccs2.5} has the following form \cite{ChenShi2021}
\begin{equation}\label{ccs2.6}
\left\{\begin{split}
-\mathcal{L}_{\delta} u_i &=f_{i}, ~~i \in \mathcal{N}_{in},\\
                  u_i                  &=g_i,      ~~i \in \mathcal{N}_{out}.
\end{split}\right.
\end{equation}

For convenience of implementation, we use the matrix form of the grid functions
\begin{equation*}
\begin{split}
U_{h}=[u_{1},u_{2},\ldots,u_{N-1},u_{\frac{1}{2}},u_{\frac{3}{2}},\ldots,u_{N-\frac{1}{2}}]^T~~~{\rm and}~~~
F_{h}=[f_{1},f_{2},\ldots,f_{N-1},f_{\frac{1}{2}},f_{\frac{3}{2}},\ldots,f_{N-\frac{1}{2}}]^T.
\end{split}
\end{equation*}

We next deal with the boundary conditions to ensure that the proposed fast AMG will have a $\mathcal{O}(N \mbox{log} N)$ complexity.
Let  the components  of  the left boundary ${}^Lg^{v}$ (or right boundary ${}^Rg^{v}$) at integer nodes  and
the left boundary ${}^Lg^{w}$ (or right boundary ${}^Rg^{w}$) at semi-integer nodes be
\begin{equation*}
\begin{split}
{}^Lg^{v}=&\left(g_{-r},g_{-r+1},\ldots,g_{-1},g_0\right)^T, ~~{}^Rg^{v}=\left(g_{N},g_{N+1},\ldots,g_{N+r}\right)^T,\\ {}^Lg^{w}=&\left(g_{-r+\frac{1}{2}},g_{-r+\frac{3}{2}},\ldots,g_{-\frac{1}{2}}\right)^T,~~{}^Rg^{w}=\left(g_{N+\frac{1}{2}},g_{N+\frac{3}{2}},\ldots,g_{N+r-\frac{1}{2}}\right)^T.
\end{split}
\end{equation*}
Let
$F^v= [f_{1},f_{2},\ldots,f_{N-1}]^T$, $F^w =[f_{\frac{1}{2}},f_{\frac{3}{2}},\ldots,f_{N-\frac{1}{2}}]^T$ and
\begin{equation*}
\begin{split}
w^{A}=&\left(a_{1},a_{2},\ldots,a_{r},0\right)^T, ~~w^{B}=\left(a_{\frac{3}{2}},a_{\frac{5}{2}},\ldots,a_{r-\frac{1}{2}},0\right)^T,\\ w^{C}=&\left(c_{0},c_{1},\ldots,c_{r}\right)^T,~~w^{D}=\left(d_{1},d_{2},\ldots,d_{r}\right)^T.
\end{split}
\end{equation*}
Denote
\begin{equation*}
{}^LF^v_{j} = \sum^{r+1}_{i=j} {}^Lg^{v}_{i}w^{A}_{r+1+j-i} + \sum^{r}_{i=j} {}^Lg^{w}_{i}w^{B}_{r+j-i},
~~~{}^RF^v_{j} = \sum^{r+2-j}_{i=1} {}^Rg^{v}_{i}w^{A}_{j-1+i} + \sum^{r+1-j}_{i=1} {}^Rg^{w}_{i}w^{B}_{j-1+i}
\end{equation*}
and
\begin{equation}\label{eq:rightside1}
F^v_{j}  =  F^v_{j} - {}^LF^v_{j},
~~~F^v_{N-j} = F^v_{N-j} - {}^RF^v_{j},~~j=1,\dots,r.
\end{equation}
Let
\begin{equation*}
\begin{split}
{}^LF^w_{j} &= \sum^{r+1}_{i=j} {}^Lg^{ v}_{i}w^{ C}_{r+1+j-i} + \sum^{r}_{i=j} {}^Lg^{ w}_{i}w^{
D}_{r+j-i},~
~~{}^LF^w_{r+1} = {}^Lg^{ v}_{r+1}w^{ C}_{r+1},\\
{}^RF^w_{j}& = \sum^{r+2-j}_{i=1} {}^Rg^{ v}_{i}w^{ C}_{j-1+i} + \sum^{r+1-j}_{i=1} {}^Rg^{ w}_{i}w^{D}_{j-1+i},~
~~{}^RF^w_{r+1} =  {}^Rg^{v}_{1}w^{ C}_{r+1}
\end{split}
\end{equation*}
and
\begin{equation}\label{eq:rightside2}
F^w_{j}  =  F^w_{j} - {}^LF^w_{j},
~~~F^w_{N-j} = F^w_{N-j} - {}^RF^w_{j},~~ j=1,\dots,r+1.
\end{equation}
From \eqref{eq:rightside1} and \eqref{eq:rightside2},
 the numerical scheme \eqref{ccs2.6} can be recast as \cite{CQS:20,ChenShi2021}
\begin{equation}\label{eq:nonsym system}
\mathcal{A}^N_{h}U_h = \eta_h\cdot F^N_{h}=\left[\begin{array}{c} F^v\\F^w \end{array}\right]~~
{\rm with}~~\mathcal{A}^N_{h} = \left[\begin{array}{cc} A & B\\ C & D \end{array}\right], \eta_h=2\delta^3/h.
\end{equation}
Note that   the above Toeplitz matrices $A\in \mathbb{R}^{{(N-1)}\times (N-1)}$, $B\in \mathbb{R}^{{(N-1)}\times N}$, $C\in \mathbb{R}^{{N}\times (N-1)}$ and
$D\in \mathbb{R}^{{N}\times N}$  are defined by
\begin{equation*}
\begin{split}
A&= {\rm toeplitz}\left(a_0,a_1,a_2,\cdots,a_{r},\textbf{0}_{1\times (N-r-2)}\right),\\
B &= {\rm toeplitz}\left(\left[a_{\frac{1}{2}},a_{\frac{3}{2}},\cdots,a_{r-\frac{1}{2}},\textbf{0}_{1 \times (N-1-r)}\right],
\left[a_{\frac{1}{2}},a_{\frac{1}{2}},a_{\frac{3}{2}},\cdots,a_{r-\frac{1}{2}},\textbf{0}_{1\times (N-r-1)}\right]\right),\\
C& = {\rm toeplitz}\left([c_0,c_0,c_1,c_2,\cdots,c_{r},\textbf{0}_{1\times (N-r-2)}],[c_0,c_1,c_2,\cdots,c_{r},\textbf{0}_{1\times(N-r-2)}]\right),\\
 D &= {\rm toeplitz}\left([d_0,d_1,d_2,\cdots,d_{r},\textbf{0}_{1\times(N-r-1)}]\right)
 \end{split}
\end{equation*}
with the coefficients
\begin{equation*}
\begin{split}
 & a_{0}=12r-2, \quad a_{m}=-2, \quad a_{r}=-1, ~~1\leq m\leq r-1,\\
 & a_{m+\frac{1}{2}}=-4, ~~0\leq m\leq r-1,
\end{split}
\end{equation*}
and
\begin{equation*}
\begin{split}
&c_{m}=-2, \quad c_{r-1}=-\frac{9}{4}, \quad c_{r}=\frac{1}{4}, ~~0\leq m\leq r-2,\\
&d_{0}=12r-4, \quad d_{m}=-4,  \quad d_{r}=-2, ~~1\leq m\leq r-1.\\
\end{split}
\end{equation*}

\subsubsection{Symmetric positive definite   block-structured dense  systems}
We observe that the discrete maximum principle is not satisfied for the above nonsymmetric indefinite system \eqref{eq:nonsym system},
which might be trickier for the stability analysis of the high-order numerical schemes \cite{DDGGTZ:20,LTTF:21}.
As a consequence, the shifted-symmetric piecewise quadratic polynomial collocation method for peridynamic /nonlocal model \eqref{ccs2.5}
has been considered in \cite{CQS:20,ChenShi2021}, which satisfies  the discrete maximum principle and has the symmetric positive definite   block-structured dense  systems.
Namely, the  shifted-symmetric  system of \eqref{ccs2.6} can be recast as
\begin{equation}\label{eq: system1}
\mathcal{A}^S_{h}U_h = \eta_h\cdot F_{h}^S~~{\rm with}~~\mathcal{A}^S_{h} = \left[\begin{array}{cc} A & B\\ B^T & \widehat{A} \end{array}\right], \eta_h=2\delta^3/h.
\end{equation}
Here the  function  $F^{S}_{h}$ is computed as done in   \eqref{eq:rightside1} and \eqref{eq:rightside2},
the Toeplitz matrices $A\in \mathbb{R}^{{(N-1)}\times (N-1)}$, $B\in \mathbb{R}^{{(N-1)}\times N}$, and
$\widehat{A}\in \mathbb{R}^{{N}\times N}$  are defined by
\begin{equation*}
\begin{split}
A &=  {\rm toeplitz}(a_0,a_1,a_2,\cdots,a_{r},\textbf{0}_{1\times(N-r-2)}),\\
B &= {\rm toeplitz}\left([a_{\frac{1}{2}},a_{\frac{3}{2}},\cdots,a_{r-\frac{1}{2}},\textbf{0}_{1\times(N-1-r)}], [a_{\frac{1}{2}},a_{\frac{1}{2}},a_{\frac{3}{2}},\cdots,a_{r-\frac{1}{2}},\textbf{0}_{1\times(N-r-1)}]\right),\\
\widehat{A} &=  {\rm toeplitz}([a_0,a_1,a_2,\cdots,a_{r},\textbf{0}_{1\times(N-r-1)}])
\end{split}
\end{equation*}
with
\begin{equation}\label{coeeq: system1}
\begin{split}
 & a_{0}=12r-2, \quad a_{m}=-2, \quad a_{r}=-1,~~1\leq m\leq r-1,\\
 & a_{m+\frac{1}{2}}=-4, ~~0\leq m\leq r-1.
\end{split}
\end{equation}

\section{Fast AMG for block-structured dense  systems}\label{sec:03}
Multigrid methods  are among the most efficient iterative methods for solving large scale systems of equations  \cite{RugeSIAM1987,Xu:17}.
However,  it is not easy to build restriction and prolongation operators when using AMG methods for nonsymmetric {\em sparse} systems  \cite{Bre2010aggregation,MaNs:19}.
To the best of our knowledge, there is no  fast AMG for block-structured dense  linear systems of the type in \eqref{eq: blocksystem},
since the   special prolongation/transfer operators are not easy to be employed with fast Fourier transform.
Here,  the simple (traditional) transfer operator are employed to handle such  block-structured dense  systems
 to ensure a fast AMG showing a $\mathcal{O}(N \mbox{log} N)$ complexity.

\subsection{Multigrid methods}
Let us firstly review the basic multigrid technique when applied to a scalar algebraic linear system, having in mind that our target is the efficient solution of the block-structured dense linear systems as those reported in \eqref{eq: blocksystem}.
Let the finest mesh points $x_i=a+ih$, $h=(b-a)/N$, $N=2^K$ with $\Omega = (a,b)$.
Define the multiple level of grids
\begin{equation*}
  \mathfrak{B}^k = \left\{x^k_{i}=a+\frac{i}{2^{k}}(b-a), i = 1:N_k \right\}~{\rm with}~ N_k = 2^k-1, k = 1:K,
\end{equation*}
where $\mathfrak{B}^k$ represents not only the grid with grid spacing $h_k=2^{K-k}h$, but also the space of vectors defined on that grid.

Given a algebraic system
 \begin{equation}\label{eq:specialsystem}
A_hu^h=b^h.
\end{equation}
We define a sequence of subsystems on different levels
$$A_ku^k=b^k, u^k\in \mathfrak{B}^k,  \quad k=1{,\ldots,}K.$$
Here $K$ is the total number of levels, with $k=K$ being the finest level, i.e., $A_K=A_h$.

The traditional  restriction operator $I^{k-1}_k:\mathds{R}^{N_{k}}\rightarrow \mathds{R}^{N_{k-1}}$ and prolongation operator $I^k_{k-1}:\mathds{R}^{N_{k-1}}\rightarrow \mathds{R}^{N_{k}}$
 are defined by \cite[pp. 438-454]{SaadInterativeSiam2003}
\begin{equation*}
  I_k^{k-1} = \frac{1}{4}\left[\begin{array}{cccccccc}
    1 & 2 & 1 &             &             &                &       &     \\
      &    & 1 & 2          & 1           &                &        &      \\
      &    &    &            & \cdots     &  \cdots     &        &       \\
      &    &    &             &             &1              &2       &1
  \end{array}\right]_{N_{k-1}\times N_k} \quad  {\rm and}\quad     I^k_{k-1} =2\left(I_k^{k-1}\right)^{T},
\end{equation*}
which  should be  convenient and efficient  for  block-structured dense  linear systems as those in \eqref{eq: blocksystem}, using AMG and fast Fourier transforms.
Let
\begin{equation*}
  \nu^{k-1} = I_k^{k-1}\nu^k~~{\rm with}~~\nu^{k-1} = \frac{1}{4}\left(\nu^k_{2i-1} + 2\nu^{k}_{2i} + \nu^{k}_{2i+1}\right),~i = 1:N_k,
\end{equation*}
and
\begin{equation*}
   \nu^{k} = I^k_{k-1}\nu^{k-1}.
\end{equation*}
It may be more useful to define the linear system by Galerkin projection in the AMG method, where the coarse grid problem is defined by
\begin{equation}\label{eq:5 Galerkin}
A_{k-1} = I_k^{k-1}A_k I^k_{k-1},
\end{equation}
and the intermediate $\left(k,k-1\right)$ coarse grid correction operator is
 \begin{equation*}\label{eq:6 correct}
   T_k =I_k - I^k_{k-1}A_{k-1}^{-1}I^{k-1}_kA_k.
 \end{equation*}
 We  choose  the damped Jacobi iterative operator
\begin{equation}\label{eq:7 Jacobi}
  S_k = I - \omega D_k^{-1}A_k\
\end{equation}
with a weighting  factor $\omega$, and $D_k$ being the  diagonal of $A_k$.

A multigrid process is intrinsically to define a sequence of operators $B_k:\mathfrak{B}^k\rightarrow\mathfrak{B}^k$, which is an approximate inverse of $A_k$ in the sense that  $\left|\left|I - B_kA_k \right|\right|$ is small, bounded away from one. The V-cycle Multigrid Algorithm  \ref{Algorithm1:V-cycle} can be seen  in \cite{XuMMEPDES1996}.
If  $k=2$, the resulting  Algorithm  \ref{Algorithm1:V-cycle} is two-grid method (TGM).
\begin{algorithm}[!ht]
\caption{ V-cycle Multigrid Algorithm: Define $B_1=A_1^{-1}$. Assume that $B_{k-1}:\mathfrak{B}^{k-1}\mapsto \mathfrak{B}^{k-1}$ is defined.
We shall now define $B_k:\mathfrak{B}^{k}\mapsto \mathfrak{B}^{k}$ as an approximate iterative solver for the equation associated with $A_ku^k=b^k$.}
\label{Algorithm1:V-cycle}
\begin{algorithmic}[1]
\STATE Pre-smooth: Let $S_{k,\omega}$ be defined by \eqref{eq:7 Jacobi} and   $u^k_0=\mathbf{0}$, for $l = 1:m_1$,
\begin{equation*}
u^k_l = u^k_{l-1} + S_{k,\omega_{pre}}(b^k - A_ku^k_{l-1})
\end{equation*}\\[1mm]
\STATE Coarse grid solution: Denote $e^{k-1}\in \mathfrak{B}^{k-1}$ as the approximate solution of the residual equation $A_{k-1}e = I^{k-1}_k(b^k - A_ku^k_{m_1})$ with the iterator $B_{k-1}$,
\begin{equation*}
  e^{k-1} =  B_{k-1}I^{k-1}_k(b^k - A_ku^k_{m_1})
\end{equation*}\\[1mm]
\STATE Correction: $u^k_{m_1+1} = u^k_{m_1} + I^k_{k-1}e^{k-1}$,
\STATE Post-smooth: for $l = m_1 + 2 :m_1 + m_2$,
\begin{equation*}
  u^k_l = u^k_{l-1} + S_{k,\omega_{post}}(b^k - A_ku^k_{l-1})
\end{equation*}\\[1mm]
\STATE Define: $B_kb^k = u^k_{m_1+ m_2}$.
\end{algorithmic}
\end{algorithm}

The basic AMG  idea for solving  the block-structured dense  linear systems in \eqref{eq: blocksystem} is the same as in the scalar case \eqref{eq:specialsystem}.
Define a sequence of block-structured subsystems on different levels
$$\mathcal{A}_ku^k=b^k, ~~  u^k\in \mathfrak{M}^k, \quad k=1{,\ldots,}K$$
with the multiple level of grids
\begin{equation*}
  \mathfrak{M}^k = \left\{x^k_{i/2}=a+\frac{i/2}{2^{k}}(b-a), i = 1:2N_k+1 \right\}~{\rm with}~ N_k = 2^k-1, k = 1:K.
\end{equation*}
Here
\begin{equation*}
\mathcal{A}_k=\left[\begin{array}{cc} A^{(k)} & B^{(k)}\\ C^{(k)} & D^{(k)} \end{array}\right],~
u^k=\left[\begin{array}{c} v^k\\ w^k   \end{array}\right],~
b^k=\left[\begin{array}{c} b_f^k \\ b_g^k  \end{array}\right],
\end{equation*}
 and $\mathcal{D}_k$ is the diagonal matrix of $\mathcal{A}_k$.
The block-structured dense V-cycle Multigrid method  is developed in Algorithm \ref{Algorithm2:structured V-cycle}.
\begin{algorithm}[!ht]
\caption{Block-structured dense V-cycle Multigrid method: Define $\mathcal{B}_1=\mathcal{A}_1^{-1}$. Assume that $\mathcal{B}_{k-1}:\mathfrak{M}^{k-1}\mapsto \mathfrak{M}^{k-1}$ is defined.
We shall now define $\mathcal{B}_k:\mathfrak{M}^{k}\mapsto \mathfrak{M}^{k}$ as an approximate iterative solver for the equation associated with $\mathcal{A}_ku^k=b^k$.}
\label{Algorithm2:structured V-cycle}
\begin{algorithmic}[1]
\STATE  Pre-smooth: Let $\mathcal{S}_{k,\omega}$ be defined by $\mathcal{S}_k = I - \omega \mathcal{D}_k^{-1} \mathcal{A}_k$ and
$\left[\begin{array}{c} v^k_0 \\ w^k_0   \end{array}\right]=\mathbf{0}$, for $l = 1:m_1$,
\begin{equation*}
\left[\begin{array}{c} v^k_l \\ w^k_l   \end{array}\right]
=\left[\begin{array}{c} v^k_{l-1} \\ w^k_{l-1}   \end{array}\right]
+ \mathcal{S}_{k,\omega_{pre}}\left(\left[\begin{array}{c} b_f^k \\ b_g^k  \end{array}\right] - \mathcal{A}_k\left[\begin{array}{c} v^k_{l-1} \\ w^k_{l-1}   \end{array}\right] \right)
\end{equation*}
\STATE  Coarse grid solution: Denote $ \mathbf{e}^{k-1}=\left[\begin{array}{c} e^{k-1}_v \\ e^{k-1}_w  \end{array}\right] \in \mathfrak{M}_{k-1}$ as the approximate solution of the residual equation
 $\mathcal{A}_{k-1}\mathbf{e} = I^{k-1}_k\left(\left[\begin{array}{c} b_f^k \\ b_g^k  \end{array}\right] - \mathcal{A}_k\left[\begin{array}{c} v^k_{m_1} \\ w^k_{m_1}   \end{array}\right]\right)$ with the iterator $\mathcal{B}_{k-1}$ an approximate inverse of $\mathcal{A}_{k-1}$,
\begin{equation*}
  \left[\begin{array}{c} e^{k-1}_u \\ e^{k-1}_v  \end{array}\right] =  \mathcal{B}_{k-1}I^{k-1}_k\left(\left[\begin{array}{c} f_k \\ g_k  \end{array}\right] - \mathcal{A}_k\left[\begin{array}{c} u^k_{m_1} \\ v^k_{m_1}   \end{array}\right]\right)
\end{equation*}
\STATE  Correction: $\left[\begin{array}{c} v^k_{m_1+1} \\ w^k_{m_1+1}   \end{array}\right]
 = \left[\begin{array}{c} v^k_{m_1} \\ w^k_{m_1}   \end{array}\right] + I^k_{k-1} \left[\begin{array}{c} e^{k-1}_v \\ e^{k-1}_w  \end{array}\right] $,
\STATE  Post-smooth: for $l = m_1 + 2 : m_1 + m_2$,
\begin{equation*}
  \left[\begin{array}{c} v^k_{l} \\ w^k_{l}   \end{array}\right] = \left[\begin{array}{c} v^k_{l-1} \\ w^k_{l-1}   \end{array}\right]
  + \mathcal{S}_{k,\omega_{post}}\left(\left[\begin{array}{c} b_f^k \\ b_g^k  \end{array}\right]
   - \mathcal{A}_k\left[\begin{array}{c} v^k_{l-1} \\ w^k_{l-1}   \end{array}\right]\right)
\end{equation*}
\STATE   Define: $\mathcal{B}_k\left[\begin{array}{c} b_f^k \\ b_g^k   \end{array}\right]  = \left[\begin{array}{c} v^k_{m_1 + m_2} \\ w^k_{m_1 + m_2}   \end{array}\right]$.
\end{algorithmic}
\end{algorithm}

\subsection{Fast Fourier transform  for block-structured dense  systems}
We first review the Toeplitz matrix and the circulant matrix by fast Fourier transform with $\mathcal{O}(n \mbox{log} n)$  complexity, which will be used later.
Let  $n \times n$ Toeplitz  matrix $T_n(c)$ be  defined by  \cite{Bottcher:05,Chan:07}
\begin{equation*}
T_n(c):=[c_{j-k}]_{j,k=1}^n=\left [ \begin{matrix}
                      c_0           &      c_{-1}             &      \cdots         &       c_{-(n-1)}       \\
                      c_{1}         &      c_{0}              &      \cdots         &       c_{-(n-2)}        \\
                     \vdots         &      \vdots             &      \ddots         &        \vdots            \\
                     c_{n-1}        &      c_{n-2}            &      \cdots         & c_{0}
 \end{matrix}
 \right ],
\end{equation*}
and the circulant matrix be the periodic cousin of Toeplitz matrix. Denote the circulant matrix $C_n$
whose first column is $\widetilde{c}=(c_0,c_1,\dots,c_{n-1})^{\rm T}$, namely,
\begin{equation*}
C_n:=\left [ \begin{matrix}
                      c_0           &      c_{n-1}      &      c_{n-2}       &      \cdots       &       c_2   &       c_1       \\
                      c_{1}         &      c_{0}        &      c_{n-1}       &      \cdots       &       c_3   &       c_{2}      \\
                      c_{2}         &      c_{1}        &      c_{0}         &      \ddots       &      \ddots &       c_{3}       \\
                      \vdots        &      \vdots       &      \ddots        &      \ddots       &      \ddots &       \vdots       \\
                      c_{n-2}       &      c_{n-3}      &      \ddots        &      \ddots       &       c_{0} &       c_{n-1}       \\
                      c_{n-1}       &      c_{n-2}      &      c_{n-3}       &      \cdots       &       c_{1} &       c_{0}
 \end{matrix}
 \right ].
\end{equation*}
Moreover, we set $\omega_n=\mbox{exp}(2\pi i/n)$ and  define the unitary Fourier matrix
\begin{equation*}
F_n=\frac{1}{\sqrt{n}}\left [ \begin{array}{llllr}
                      1          &          1           &               1       &      \cdots         &               1             \\
                      1          &      \omega_n        &      \omega_n^{2}     &      \cdots         &       \omega_n^{n-1}         \\
                      1          &      \omega_n^{2}    &      \omega_n^{4}     &      \cdots         &       \omega_n^{2(n-1)}       \\
                      \vdots     &      \vdots          &      \vdots           &      \ddots         &        \vdots                  \\
                      1          &      \omega_n^{n-1}  &      \omega_n^{2(n-1)}&      \cdots         &        \omega_n^{(n-1)(n-1)}
\end{array}
 \right ]
\end{equation*}
with $i$ being the imaginary unit.
Therefore, a circulant matrix can be diagonalized by the  Fourier matrix $F_n$, i.e.,
\begin{equation*}
C_n=F_n^{-1}\mbox{diag}(\sqrt{n}F_n\widetilde{c} )F_n=F_n^{*}\mbox{diag}(\sqrt{n}F_n\widetilde{c} )F_n,
\end{equation*}
where  $\mbox{diag}(\sqrt{n}F_n\widetilde{c} )$  is a diagonal matrix holding the eigenvalues of $C_n$.
Moreover, for any given vector $v$, we can determine \cite{ChanNg:96}
\begin{equation*}
C_nv={\rm ifft}\left( {\rm fft}(\widetilde{c}).*{\rm fft}(v)   \right)
\end{equation*}
with $\mathcal{O}(n \mbox{log} n)$ operations by the fast Fourier transform (FFT) of the first column $\widetilde{c} $ of $C_n$.

Then, for any $n$-by-$1$   vector $\bf x$, the multiplication $T_n\bf x$ can also be computed by FFTs with  the
computational count of $\mathcal{O}\left(n \log  n\right)$ arithmetic operations \cite[p.\,12]{Chan:07}.
More concretely, we take a $2n$-by-$2n$
circulant matrix $C_n$ with $T_n$ embedded inside as follows:
\begin{equation}\label{ccs3.4}
\left[\begin{array}{cc}T_{n} & \ast\\ \ast & T_{n}\end{array}\right]\left[\begin{array}{c}\bf x\\ \bf 0\end{array}\right]=\left[\begin{array}{c}T_{n}\bf x\\ \ddag\end{array}\right]
~~~~{\rm with}~~C_n=\left[\begin{array}{cc}T_{n} & \ast\\ \ast & T_{n}\end{array}\right].
\end{equation}
\subsubsection{Fast Fourier transform  for block-structured dense  systems at the finest level}\label{ssub1}
Let us  consider  the fast Fourier transform algorithm   for block-structured dense  systems \eqref{eq: blocksystem} at the finest level, namely,
\begin{equation}\label{ccs3.5}
\mathcal{A}_hu=\left(\begin{array}{cc} A & B\\ C & D \end{array}\right)\left(\begin{array}{c} v \\ w  \end{array}\right)=\left(\begin{array}{c} Av+Bw \\ Cv+Dw  \end{array}\right)
\end{equation}
with    Toeplitz matrices $A\in \mathbb{R}^{M\times M}$,  $B\in \mathbb{R}^{M\times N}$,  $C\in \mathbb{R}^{N\times M}$ and $D\in \mathbb{R}^{N\times N}$, $M<N$.

It is well known that most of the early works on matrix-vector multiplication with Toeplitz algebraic system  were focused on squared matrices by Fast fourier transform (FFT) \cite{Chan:07,ChanNg:96}.
However, the considered technique can not be  directly applied  to block-structured dense systems of the form \eqref{ccs3.5}, since the related structures contain  rectangular the matrices $B$ and $C$.

We next design/develop the FFT algorithm  for the rectangular matrices $B$ and $C$ in \eqref{ccs3.5} based on the idea of \cite{CCNWL:20,Chan:07}, which  realizes  the computational count of $\mathcal{O}(N \log  N)$ arithmetic operations and the required storage $\mathcal{O}(N)$.
Let the rectangular matrix $B\in \mathbb{R}^{M\times N}$  be given  by
\begin{equation*}
\begin{split}
B=\left [ \begin{matrix}
b_{0}               & b_{1}              & \cdots            &  b_{J}          & \cdots      &  \cdots              & \cdots            & b_{N-1}      \\
b_{-1}               & b_{0}            & b_{1}             &     \ddots       & \ddots      &                &             & \vdots              \\
b_{ -2}              & b_{-1}           & b_{0}             &     \ddots       & \ddots      & \ddots      &               &  \vdots       \\
\vdots               & \ddots             &  \ddots           &     \ddots       & \ddots      &  \ddots      &  \ddots               & \vdots     \\
b_{-M+1}          &\cdots              & b_{ -2}            & b_{-1}            & b_{0}      & b_{1}        &  \cdots     &   b_{J}
 \end{matrix}
 \right ]_{M \times N}~~{\rm with}~~J=N-M.
\end{split}
\end{equation*}
We embed the matrix $B\in \mathbb{R}^{M\times N}$ into the following $N$-by-$N$ Toeplitz matrix $B_T$,
\begin{equation*}
\begin{split}
B_T=\left [ \begin{matrix}
b_{0}               & b_{1}              & \cdots            &  b_{J}          & \cdots      &  \cdots              & \cdots            & b_{N-1}      \\
b_{-1}               & b_{0}            & b_{1}             &     \ddots       & \ddots      &                &             & \vdots              \\
b_{ -2}              & b_{-1}           & b_{0}             &     \ddots       & \ddots      & \ddots      &               &  \vdots       \\
\vdots               & \ddots             &  \ddots           &     \ddots       & \ddots      &  \ddots      &  \ddots               & \vdots     \\
b_{-M+1}          &\cdots              & b_{ -2}          & b_{-1}            & b_{0}      & b_{1}        &  \cdots     &   b_{J}     \\ \hline
0                     &\ddots               &                     &  \ddots           &     \ddots          &   \ddots              &     \ddots           &  \vdots     \\
\vdots              &\ddots               &  \ddots            &                     &  \ddots    &  \ddots               &    \ddots            &   b_{1}     \\
0            &\cdots     &0                     &b_{-M+1}       & \cdots      & b_{ -2}     &  b_{-1}    &   b_{0}     \\
 \end{matrix}
 \right ]_{N \times N}
\end{split}.
\end{equation*}
Using  \eqref{ccs3.4} and fast Fourier transform, it implies that $Bw$  is the first $M$-rows of  $B_Tw$ with
\begin{equation*}
\left[\begin{array}{cc}B_T & \ast\\ \ast & B_T\end{array}\right]\left[\begin{array}{c} w\\ \bf 0\end{array}\right]=\left[\begin{array}{c}B_Tw\\ \ddag\end{array}\right].
\end{equation*}
 Let the rectangular matrix $C\in \mathbb{R}^{N\times M}$ be defined by
\begin{equation*}
\begin{split}
C=\left [ \begin{matrix}
c_{0}               & c_{1}               &     \cdots       &c_{M-1}         \\
c_{-1}              &\ddots               &     \ddots    & \vdots     \\
\vdots               & \ddots              &     \ddots    & c_{1}       \\
c_{-J}               & \cdots             &     c_{-1}       &c_{0}       \\
 \vdots                &   \ddots            &      \ddots               & c_{-1}     \\
 \vdots              &                 & \ddots                      & \vdots      \\
c_{-N+1}         & \cdots                 & \cdots      &c_{-J}
 \end{matrix}
 \right ]_{N\times M}~~{\rm with}~~J=N-M.
\end{split}
\end{equation*}
Similar, we embed matrix $C\in \mathbb{R}^{N\times M}$ into the following $N$-by-$N$ Toeplitz matrix $C_T$,
\begin{equation*}
\begin{split}
C_T=\left [ \begin{matrix}
c_{0}               & c_{1}               &     \cdots       &c_{M-1}         \\
c_{-1}              &\ddots               &     \ddots    & \vdots     \\
\vdots               & \ddots              &     \ddots    & c_{1}       \\
c_{-J}               & \cdots             &     c_{-1}       &c_{0}       \\
 \vdots                &   \ddots            &      \ddots               & c_{-1}     \\
 \vdots              &                 & \ddots                      & \vdots      \\
c_{-N+1}         & \cdots                 & \cdots      &c_{-J}
 \end{matrix}
 \right |
\left . \begin{matrix}
 0     &\cdots     &0          \\
c_{M-1}      &\ddots     &\vdots  \\
 \vdots               &\ddots     &0      \\
  c_{1}      &     \cdots         &c_{M-1}      \\
 c_0    &  \ddots    &\vdots     \\
 \ddots     &\ddots     &c_{1}    \\
 \cdots     &c_{-1}     &c_{0}
 \end{matrix}
 \right ]_{N \times N}.
\end{split}
\end{equation*}
From   \eqref{ccs3.4} and fast Fourier transform, we have
\begin{equation*}
Cv=C_T{\bf v}~~{\rm with}~~\left[\begin{array}{cc}C_T & \ast\\ \ast & C_T\end{array}\right]\left[\begin{array}{c} {\bf v}\\ \bf 0\end{array}\right]=\left[\begin{array}{c}C_T{\bf v}\\ \ddag\end{array}\right],
{\bf v}=\left[\begin{array}{c}  v\\ {\bf 0}_{J\times 1}\end{array}\right].
\end{equation*}

\subsubsection{Fast Fourier transform  for block-structured dense  systems at the coarser level}
In fact,  embedding rectangular matrices $B$ and $C$ in \eqref{ccs3.5} into  the square Toeplitz matrices in subsubsection \ref{ssub1}
is invalid, when using the Galerkin projection \eqref{eq:5 Galerkin} in the AMG method,
although this technique can be applied to compute the finest level for a fast AMG,
since it does not keep block-structured Toeplitz properties.
Next, we shall design the  FFT for  the block-structured matrix $\mathcal{A}_k$ multiplies vector   at the coarser level, i.e.,
\begin{equation*}
{\mathcal{A}}_k\left[\begin{array}{c} u^k \\ v^k  \end{array}\right]=\left[\begin{array}{c}
A^{(k)} u^k + B^{(k)} v^k\\[2mm]
C^{(k)} u^k + D^{(k)} v^k
\end{array}\right].
\end{equation*}
 In AMG, the coarse problem at the  level $k<K$ is typically defined using the Galerkin approach, i.e., the coefficient matrix on the coarser grid can be computed by
\begin{equation}\label{eq:block Galerkin}
\mathcal{A}_{k-1} = I_k^{k-1}\mathcal{A}_k I^k_{k-1}.
\end{equation}
More concretely,
\begin{equation*}
\mathcal{A}_{k-1} =\left[\begin{array}{cc}
A^{(k-1)} & B^{(k-1)} \\[2mm]
C^{(k-1)} & D^{(k-1)}
\end{array}\right]_{(2N_{k-1}+1)\times(2N_{k-1}+1)}=
I_k^{k-1}\left[\begin{array}{cc}
A^{(k)} & B^{(k)} \\[2mm]
C^{(k)} & D^{(k)}
\end{array}\right]_{(2N_k+1)\times(2N_k+1)}I^k_{k-1}
\end{equation*}
with  $I_k^{k-1}\in \mathbb{R}^{N_k\times(2N_k+1)} $ in \eqref{eq:5 Galerkin}.

It should be noted that $A^{(k)}, B^{(k)}, C^{(k)}, D^{(k)}$ are Toeplitz matrices if $k=K$,
which corresponds to  the block-structured dense system  \eqref{ccs3.5} at the fineast level with the computational count of $\mathcal{O}(N \log  N)$ arithmetic operations.
However, there is major  difference for $\mathcal{A}_{k-1}$  at the coarser level, since it does not keep block-structured Toeplitz properties, see Example \ref{ex1}.
\begin{example}\label{ex1}
Choose   the unit matrices     $A^{(k)}\in \mathbb{R}^{7\times 7}$, $D^{(k)}\in \mathbb{R}^{8\times 8}$ and the
rectangular ones  matrices $B^{(k)}\in \mathbb{R}^{7\times 8}$, $C^{(k)}\in \mathbb{R}^{8\times 7}$
(all the  elements are $1$). Using the Galerkin approximation \eqref{eq:block Galerkin}, we deduce
\begin{equation*}
\left[\begin{array}{cc}
A^{(k-1)} & B^{(k-1)} \\[2mm]
C^{(k-1)} & D^{(k-1)}
\end{array}\right]=I_k^{k-1}\mathcal{A}_k I^k_{k-1}
=\frac{1}{8}\left[\begin{array}{ccc }
     6  &   1  &   0    \\
     1  &   6  &   1   \\
     0  &   1  &   6    \\ \hline
    12  &  12  &  13    \\  \hline
    16  &  16 &   16    \\
    16  &  16 &   16    \\
    16  &  16 &   16   \\
\end{array}\right|
\left.\begin{array}{c }
       12  \\
       12   \\
       13   \\ \hline
       12    \\  \hline
       5    \\
       4    \\
        4    \\
\end{array}\right|
\left.\begin{array}{ccc }
         16  &  16  &  16 \\
         16  &  16  &  16\\
         16  &  16   & 16 \\ \hline
          5  &   4   &  4 \\  \hline
         6  &   1   &  0 \\
           1  &   6   &  1 \\
            0  &   1  &   6 \\
\end{array}\right].
\end{equation*}
\end{example}

We next design fast Fourier transform  for block-structured dense  systems at the coarser level in AMG. Let
\begin{equation}\label{eq: split01}
\mathcal{A}_{k} =
\left[\begin{array}{ccc}
     A^{(k)}                      &\textbf{0}       & \bar{B} ^{(k)}              \\
   \textbf{0 }'                   &  0                  & \textbf{0 }'      \\
    {\bar{C}^{(k)}}           &  \textbf{0}      & \bar{D}^{(k)}
\end{array}\right]_{(2N_k+1)\times(2N_k+1)} +
\left[\begin{array}{ccc}
     \textbf{0}                & p^{(k)}                & \textbf{0}        \\
     {q^{(k)}}                 &  o^{(k)}               & {\zeta^{(k)}}                  \\
    \textbf{0}                 &  \xi^{(k)}              & \textbf{0}
\end{array}\right]_{(2N_k+1)\times(2N_k+1)},
\end{equation}
where    $ A^{(k)},\bar{B} ^{(k)} , {\bar{C}^{(k)}},\bar{D}^{(k)}$ are the square matrices with
\begin{equation*}
  B^{(k)} = \left[\begin{array}{cc} p^{(k)} & \bar{B}^{(k)}  \end{array}\right],
  ~C^{(k)}=\left[\begin{array}{c} q^{(k)} \\  [1mm]
  \bar{C}^{(k)}  \end{array}\right],
  ~D^{(k)}=\left[\begin{array}{cc} o^{(k)} & \zeta^{(k)} \\ [1mm]
   \xi^{(k)} &  \bar{D}^{(k)}  \end{array}\right].
\end{equation*}
The symbol $o^{(k)}$ is a real  number,   $0$ denotes a zero number,  and the bold $\textbf{0}$ denotes
a zero matrix/vector with the corresponding size.
The coefficients  $p^{(k)},\xi^{(k)}$ denote the column vectors and  ${q^{(k)}}, {\zeta^{(k)}}$ denote the row vectors.
We may call it the \emph{Cross-splitting} technique, since we main focus on the fast Fourier transform for  Cross-type matrix in \eqref{eq: split01}.

From \eqref{eq: split01}, it yields
\begin{equation*}\label{eq:mat_vec multiply split}
\begin{split}
{\mathcal{A}}_k\left[\begin{array}{c} v^k \\ w^k  \end{array}\right]
=&\left(\left[\begin{array}{ccc}
     A^{(k)}                      &\textbf{0}       & \bar{B} ^{(k)}              \\
   \textbf{0 }'                   &  0                  & \textbf{0 }'      \\
    {\bar{C}^{(k)}}           &  \textbf{0}      & \bar{D}^{(k)}
\end{array}\right]+
\left[\begin{array}{ccc}
     \textbf{0}                & p^{(k)}                & \textbf{0}        \\
     {q^{(k)}}                 &  o^{(k)}               & {\zeta^{(k)}}                  \\
    \textbf{0}                 &  \xi^{(k)}              & \textbf{0}
\end{array}\right]\right)
\left[\begin{array}{c} v^k \\ w^k_o \\ \bar{w}^k  \end{array}\right]\\
=&\left[\begin{array}{c}
     A^{(k)}v^k  + \bar{B} ^{(k)}\bar{w}^k              \\
                        0                                                 \\
    {\bar{C}^{(k)}}v^k  +\bar{D}^{(k)}\bar{w}^k
\end{array}\right]+
\left[\begin{array}{ccc}
                                       p^{(k)}w^k_o                                                      \\
    q^{(k)}v^k + o^{(k)}w^k_o +\zeta^{(k)}\bar{w}^k                \\
                                       \xi^{(k)}w^k_o
\end{array}\right]~~{\rm with}~~w^k=\left[\begin{array}{c} w^k_o \\ \bar{w}^k  \end{array}\right].
\end{split}
\end{equation*}
 Obviously, since $A^{(k)}$, $\bar{B} ^{(k)}$, ${\bar{C}^{(k)}}$, $\bar{D}^{(k)}$ are Toeplitz matrices, the computation of $A^{(k)}v^k$, $\bar{B} ^{(k)}\bar{w}^k$, ${\bar{C}^{(k)}}v^k$, and $\bar{D}^{(k)}\bar{w}^k$ by FFT needs $\mathcal{O}(N_k{\rm log}N_k)$, and with required storage $\mathcal{O}(N_k)$. For  the Cross  matrix, see $p^{(k)}$, $q^{(k)}$, $\zeta^{(k)}$, $\xi^{(k)}$ and $o^{(k)}$,
 results in   $\mathcal{O}(N_k)$  complexity and   storage  operations.

Let the stiffness matrix of the coarser level be
 \begin{equation}\label{eq: split 02}
\mathcal{A}_{k-1} =
\left[\begin{array}{ccc}
     A^{(k-1)}                      &\textbf{0}       & \bar{B} ^{(k-1)}              \\
   \textbf{0 }'                   &  0                  & \textbf{0 }'      \\
    {\bar{C}^{(k-1)}}           &  \textbf{0}      & \bar{D}^{(k-1)}
\end{array}\right]_{N_k\times N_k} +
\left[\begin{array}{ccc}
     \textbf{0}                & p^{(k-1)}                & \textbf{0}        \\
     {q^{(k-1)}}              &  o^{(k-1)}               & {\zeta^{(k-1)}}                  \\
    \textbf{0}                 &  \xi^{(k-1)}              & \textbf{0}
\end{array}\right]_{N_k\times N_k}.
\end{equation}
Then using  \eqref{eq: split01}, \eqref{eq: split 02}, and a Galerkin approach in \eqref{eq:block Galerkin}, it is easy to obtain   the following results.
\begin{lemma}\label{lm3.1}
Let $A^{(k)}=\{a_{i,j}^{(k)}\}_{i,j=1}^{N_k}$ with $a_{i,j}^{(k)}=a_{j-i}^{(k)}$ be a Toeplitz matrix in \eqref{eq: split01}. Then the elements of
$A^{(k-1)}$ in \eqref{eq: split 02} can be computed by
 \begin{equation*}
 8 a_0 ^{(k-1)} =a_{-2}  ^{(k)}  + 4a_{-1} ^{(k)} + 6a_0 ^{(k)} + 4a_1^{(k)} + a_2^{(k)},
\end{equation*}
and
\begin{equation*}
\begin{split}
 8 a_i ^{(k-1)} &= a_{2i-2}  ^{(k)}  + 4a_{2i-1} ^{(k)} + 6a_{2i} ^{(k)} + 4a_{2i+1}^{(k)} + a_{2i+2}^{(k)},\\
8  a_{-i} ^{(k-1)} &= a_{-2i-2}  ^{(k)}  + 4a_{-2i-1} ^{(k)} + 6a_{-2i} ^{(k)} + 4a_{-2i+1}^{(k)} + a_{-2i+2}^{(k)},~~~i\geq 1.
  \end{split}
\end{equation*}
Moreover, for the Cross matrix, the following terms can be defined
\begin{equation*}
\begin{split}
 8 p^{(k-1)}_i& = \left(a^{(k)}_{N_k-2i} + 2a^{(k)}_{N_k-2i-1} + a^{(k)}_{N_k-2i-2 }\right) + 2\left( p^{(k)}_{2i-1} + 2p^{(k)}_{2i} + p^{(k)}_{2i+1}\right) \\
                &\quad + \left(b^{(k)}_{-2i+2} + 2b^{(k)}_{-2i+1} + b^{(k)}_{-2i}\right), \\
8 q^{(k-1)}_i& = \left(a^{(k)}_{-N_k+2i} + 2a^{(k)}_{-N_k+2i+1} + a^{(k)}_{-N_k+2i+2 }\right) + 2\left( q^{(k)}_{2i-1} + 2q^{(k)}_{2i} + q^{(k)}_{2i+1}\right)\\
 &\quad + \left(c^{(k)}_{2i-2} + 2c^{(k)}_{2i-1} + c^{(k)}_{2i}\right),\\
 8  o^{(k-1)}&=\left(a^{(k)}_{0} + 2p^{(k)}_{N_k-1} + b^{(k)}_{-N_k+2 }\right) + 2\left( q^{(k)}_{N_k-1} + 2o^{(k)} + \zeta^{(k)}_{1}\right) +\left( c^{(k)}_{N_k-2} + 2\xi^{(k)}_{1} + d^{(k)}_{0}\right),
   \end{split}
\end{equation*}
and
\begin{equation*}
\begin{split}
 8 \xi^{(k-1)}_i &=\left(c^{(k)}_{N_k-2i} + 2c^{(k)}_{N_k-2i-1} + c^{(k)}_{N_k-2i-2 } \right)+ 2\left( \xi^{(k)}_{2i-1} + 2\xi^{(k)}_{2i}+ \xi^{(k)}_{2i+1}\right)\\
  &\quad  +\left(  d^{(k)}_{-2i+2} + 2d^{(k)}_{-2i+1} + d^{(k)}_{-2i}\right), \\
8\zeta^{(k-1)}_i &=\left(b^{(k)}_{-N_k+2i} + 2b^{(k)}_{-N_k+2i+1} + b^{(k)}_{-N_k+2i+2 }\right) + 2\left( \zeta^{(k)}_{2i-1} + 2\zeta^{(k)}_{2i} + \zeta^{(k)}_{2i+1}\right)\\
&\quad + \left(d^{(k)}_{2i-2} + 2d^{(k)}_{2i-1} + d^{(k)}_{2i}\right).
\end{split}
\end{equation*}
\end{lemma}

\subsection{The operation count and storage requirement}
We now study  the computation  complexity by the fast Fourier transform and the required storage for the block-structured dense system \eqref{ccs3.5} in AMG, arising from the nonlocal problems in Section 2.

From \eqref{ccs3.5}, we know that the  matrix $\mathcal{A}_h$ is a block-structured Toeplitz-like system.
Then, we only need to store the first (last) column, first (last) row and  principal diagonal  $\mathcal{A}_h$,
which have $\mathcal{O}(N)$ parameters, instead of the full matrix $\mathcal{A}_h$ with $N^2$ entries.
From Lemmas \ref{lm3.1}, we know that $\{A_k\}$ is a Toeplitz-like-plus-Cross  matrix  with the  sizes $2^{k-K}\mathcal{O}(N)$ storage.
Adding these terms together, we have
\begin{equation*}
  \mbox{Storage}=\mathcal{O}(N) \cdot \left( 1+\frac{1}{2}+\frac{1}{2^2}+\ldots+\frac{1}{2^{K-1}} \right)=\mathcal{O}(N).
\end{equation*}

Regarding the computational complexity, the matrix-vector product associated with the matrix $\mathcal{A}_h$ is a discrete convolution.
On the other hand the cost of a direct product is $O(N)$ for the  Cross  matrix, while the cost of using the FFT would lead to $O(N\log(N)$  for
computations involving a dense Toeplitz matrix as the one in \eqref{ccs3.4}.
Thus, the total per V-cycle AMG operation count is
\begin{equation*}
\mbox{Operation count}= \mathcal{O}( N\log N) \cdot \left( 1+\frac{1}{2}+\frac{1}{2^2}+\ldots+\frac{1}{2^{K-1}} \right)=\mathcal{O}( N\log N).
\end{equation*}

\section{Convergence of TGM for block-structured dense system \eqref{eq: system1}}
The  two-grid method (TGM) is rarely used in practice since the coarse grid operator may still be too
large to be solved exactly.  However, from a theoretical point of view, its study is useful as first step for evaluating the
MGM convergence speed, whose analysis usually begins from that of the TGM \cite{Arico:04,Arico:07,Pang:12,RugeSIAM1987,SaadInterativeSiam2003,Serra:02}.
In the following we  consider the convergence of the TGM for symmetric block-structured dense system \eqref{eq: system1}.
 Since the matrix $\mathcal{A}:=\mathcal{A}_h:=\mathcal{A}_K$ is symmetric positive definite, we can define
\begin{equation*}
  (u,v)_\mathcal{D} = (\mathcal{D}u,v),\ (u,v)_\mathcal{A }= (\mathcal{A}u,v),\ (u,v)_{\mathcal{A}\mathcal{D}^{-1}\mathcal{A}}=(\mathcal{A}u,\mathcal{A}v)_{\mathcal{D}^{-1}}
\end{equation*}
with $\mathcal{D}$ is the diagonal matrix of $\mathcal{A}$.
\begin{lemma}\cite{XuMMEPDES1996}\label{xu96}
Let $\mathcal{A}_K$ be a symmetric positive definite. Then

(1) the Jacobi method converges if and only if $2\mathcal{D}_K-\mathcal{A}_K$ is symmetric positive definite;

(2) the Damped Jacobi method converges if and only if $0<\omega<2/\lambda_{\max}(\mathcal{D}_K^{-1}\mathcal{A}_K)$.
\end{lemma}
\begin{lemma} \cite[p.\,84]{RugeSIAM1987}\label{lemma4.7}
Let $\mathcal{A}_K$ be a symmetric positive definite. If $\eta\leq\omega(2-\omega \eta_0)$ with $\eta_0\geq\lambda_{\max}(\mathcal{D}_K^{-1}\mathcal{A}_K) $,  then the damped Jacobi iteration
with relaxation parameter $0 < \omega < 2/\eta_0 $ satisfies
\begin{equation}\label{{ccs4.2}}
 ||\mathcal{S}_K\nu^K||_{\mathcal{A}_K }^2 \leq  ||\nu^K||_{\mathcal{A}_K }^2 - \eta ||\mathcal{A}_K\nu^K||_{\mathcal{D}_K^{-1}}^2 \quad  \forall \nu^K \in \mathfrak{M}^{K}.
\end{equation}
\end{lemma}

\begin{lemma}\cite[p.\,89]{RugeSIAM1987}\label{lemma4.8}
Let $\mathcal{A}_K$ be a symmetric positive definite matrix and $S_K$ satisfies (\ref{{ccs4.2}}) and
\begin{equation}\label{{ccs4.3}}
   \min_{\nu^{K-1} \in \mathfrak{M}^{K-1} }||\nu^{K}-I_{K-1}^K\nu^{K-1}||_{\mathcal{D}_K}^2\leq \mu ||\nu^K||_{\mathcal{A}_K}^2 \quad  \forall \nu^K \in \mathfrak{M}^{K}
\end{equation}
with  $\mu>0$ independent of $\nu^K$. Then, $\mu\geq \eta>0$ and the convergence factor of TGM  satisfies
\begin{equation*}
||\mathcal{S}_K\mathcal{T}_K||_{\mathcal{A}_K}\leq \sqrt{1-\eta/\mu }\quad \forall \nu^K \in \mathfrak{M}^{K}.
\end{equation*}
\end{lemma}

Next we need to check the smoothing condition \eqref{{ccs4.2}} and approximation property \eqref{{ccs4.3}}, respectively.
We will use notion called weakly diagonal dominant, if the diagonal element of a matrix is at least as large as the sum  of absolute value of  off-diagonal elements in the same row or column.

\begin{lemma}\cite[p. 23]{VargaSpringer2000}\label{swlemma1}
Let  $\mathcal{A}_K\in \mathbb{R}^{N\times N}$ be a symmetric matrix.
If $\mathcal{A}_k$ is a  strictly diagonally dominant or irreducibly weakly  diagonally dominant matrix with positive real diagonal entries, then $\mathcal{A}_K$ is positive definite.
\end{lemma}
\begin{lemma} \label{adddlemma4.1}
Let $\mathcal{A}_K:=\mathcal{A}^S_{h}$ be defined by  \eqref{eq: system1}.
Then $\mathcal{A}_K$ is a  weakly diagonally dominant symmetric matrix with positive entries on the diagonal and nonpositive off-diagonal entries.
\end{lemma}
\begin{proof}
From \eqref{coeeq: system1}, we have
\begin{equation*}
a_{0}>0, \quad a_{m}<0,~~1\leq  m\leq r \quad {\rm and} \quad a_{m+\frac{1}{2}}<0,~~0\leq m\leq r-1,
\end{equation*}
and
\begin{equation*}
a_0+2\sum_{m=1}^{r}a_m+2\sum_{m=0}^{r-1}a_{m+\frac{1}{2}}=0.
\end{equation*}
The proof is completed.
\end{proof}

We known that  the matrix  $\mathcal{A}_K$ in \eqref{eq: system1} is  reducible, since its graph is not connected \cite[p. 91]{SaadInterativeSiam2003},
which  may lead to the matrix $\mathcal{A}_K$ is  singular or semi-positive definite, see     Lemmas  \ref{swlemma1} and \ref{adddlemma4.1}.
\begin{lemma} \label{cdlemma4.2}
Let $\mathcal{A}_K:=\mathcal{A}^S_{h}$ be defined by  \eqref{eq: system1}.
Then $\mathcal{A}_K$ is a symmetric positive definite matrix with positive entries on the diagonal and nonpositive off-diagonal entries.
\end{lemma}
\begin{proof}
Let $L_{N}={\rm tridiag}(-1,2,-1)\in \mathbb{R}^{N\times N}$ be the  discrete Laplacian operator.
Let
\begin{equation}\label{{ccs4.4}}
\mathcal{A}_K=\mathcal{A}_{\rm res}-a_1\mathcal{A}_{\rm main}
 ~~{\rm with}~~\mathcal{A}_{\rm main}=
\left[\begin{array}{cc}
L_{N-1} & O        \\
	O & L_{N}
\end{array}\right].
\end{equation}
We can check  the matrix  $-a_1\mathcal{A}_{\rm main}$ with $a_1<0$ is  positive definite.
We next prove   $\mathcal{A}_{\rm res}$ is a semi-positive definite matrix.
From \eqref{{ccs4.4}}, it implies that
the principal diagonal elements are positive and the off-diagonal are non-positive of the matrix $\mathcal{A}_{\rm res}$.
Then  $\mathcal{A}_{\rm res}$ is still a diagonally dominant symmetric matrix.
Thus, $\mathcal{A}_{\rm res}$ is a semi-positive matrix by the Gerschgorin circle theorem  \cite[p.\,388]{Horn:13}.
The proof is completed.
\end{proof}
\begin{lemma}\label{adddlemma4.1220}\label{lem4.7}
Let the discrete Laplacian-like operators  $\{L^M_j\}_{j=1}^{M-1}\in\mathds{R}^{M\times M}$ and discrete  block-structured Laplacian  operators  be, respectively,  defined by
\begin{equation*}
L^M_{j}=\left [ \begin{matrix}
2      &\overset{j-1 \text{  zeros }}{\overbrace{  \cdots }} &   -1       &           &             \\
\ddots &      \ddots                                                    &  \ddots    &  \ddots   &              \\
 -1    &     \ddots                                                &  \ddots    &  \ddots   &      -1       \\
       &             \ddots                                        &  \ddots    &   \ddots  &     \ddots     \\
       &                                                           &    -1      &   \ddots  &      2
 \end{matrix}
 \right ]_{M \times M}
 ~~{\rm and}~~
 \mathcal{L}_j = \left[ \begin{array}{cc}  {L}^M_j  & ~ \mathbf{0}   \\
                                                           \mathbf{0}   &~  L^N_j  \end{array}\right].
\end{equation*}
Then the matrices
$$\left(\frac{2}{l}\sum\limits_{j=1}^lL^M_{j}\right)-L^M_{1},~ ~~~  \left(\frac{2}{l}\sum\limits_{j=1}^l \mathcal{L}_j\right)-  \mathcal{L}_1
~~{\rm and}~~2 \mathcal{L}_1- L^{M+N}_1$$
 are  positive definite.
\end{lemma}
\begin{proof}
The first results of this lemma can be seen in Lemma 3.10 of \cite{ChenDeng2017},
which implies that  the second one is also satisfied,  since
\begin{equation*}
\begin{split}
\left( \frac{2}{l}\sum\limits_{j=1}^l \mathcal{L}_j\right)-  \mathcal{L}_1
 &=\frac{2}{l}\sum^l_{j=1}\left[ \begin{array}{cc}  {L}^M_j  & ~ \mathbf{0}   \\
                            \mathbf{0}   &~  L^N_j  \end{array}\right]
 -\left[ \begin{array}{cc}  {L}^M_1  & ~ \mathbf{0}   \\
                            \mathbf{0}   &~  L^N_1  \end{array}\right]\\
& =\left[ \begin{array}{cc} \left( \frac{2}{l}\sum^l_{j=1}{L}^M_j\right)-{L}^M_1  & ~ \mathbf{0}   \\
                            \mathbf{0}   &~  \left(\frac{2}{l}\sum^l_{j=1}L^N_j\right)-L^N_1  \end{array}\right].
\end{split}
\end{equation*}
On the other hand, we can check that $2 \mathcal{L}_1- L^{M+N}_1$  is an irreducible and  weakly diagonally dominant symmetric matrix,
which  means that it is positive definite by   Lemma \ref{swlemma1}.
The proof is completed.
\end{proof}

\begin{remark}
In order to understand the spectral features of the matrices considered in Lemma  \ref{lem4.7}, we can adopt the analysis via the related generating functions, since all the matrices in Lemma  \ref{lem4.7} are of real symmetric banded Toeplitz type, so the admit real-valued trigonometric polynomials as generating functions (see \cite{Serra:02} and references therein): on the other hand the matrices $\mathcal{L}_j$ are block diagonal and hence their spectral analysis reduces to the Toeplitz setting.  For instance, according to the notation in \cite{Serra:02}, $L^M_{j}=T_M(2-2\cos(j\theta))$ that is the function $2-2\cos(j\theta)$ is the generating function of $L^M_{j}$. From classical results, we know that $T_M(f)$ is positive definite for any matrix-size $M$ if $f$ is essentially bounded and nonnegative, with positive essential supremum. In the present setting,  the maximum of $2-2\cos(j\theta)$ is 4 and its minimum is zero and hence $L^M_{j}=T_M(2-2\cos(j\theta))$ is positive definite. Not only this: if the nonnegative generating function $f$ has a unique zero of order $\alpha>0$ then the minimal eigenvalue of $T_M(f)$ is positive converges monotonically to zero as $c/M^\alpha$ with $c$ depending on the second derivative of $f$ at the zero if it is positive. Based on these tools, we can deduce that  $L^M_{j}$ is positive definite, has minimal eigenvalue positive converging monotonically to zero as $c_j/M^2$ with positive $c_j$ independent of $M$, and $c_j$ related to the second derivative of $2-2\cos(j\theta)$ at $\theta=0$.

Of course, by linearity, the Toeplitz matrix $\left(\frac{2}{l}\sum\limits_{j=1}^lL^M_{j}\right)-L^M_{1}$ has generating function given by
$f_l(\theta)\equiv \left(\frac{2}{l}\sum\limits_{j=1}^l\textcolor{black}{ \left[2-2\cos(j\theta)\right]} \right)-(2-2\cos(\theta))$
and hence by studying this generating function, we deduce that
$$\left(\frac{2}{l}\sum\limits_{j=1}^lL^M_{j}\right)-L^M_{1}$$
is positive definite, has minimal eigenvalue positive converging monotonically to zero as $c/M^2$ with positive $c$ independent of $M$.  Hence its condition number grow exactly as $M^2$ and since the related generating function has a unique zero of order $2$ at $\theta=0$, there is a formal justification in using standard projectors and restriction operators, like those employed in the standard AMG, for the classical discrete Laplacian (see \cite {Fiorentino:96,Arico:07,Serra:02} and references therein).
\end{remark}

\begin{lemma}\label{lem 4.6 smooth}
Let $\mathcal{A}_K:=\mathcal{A}^S_{h}$ be defined by  \eqref{eq: system1}.
Then the damped Jacobi iteration
with relaxation parameter $0 < \omega \leq 1 $ satisfies
\begin{equation*}
 ||S_K\nu^K||_{\mathcal{A}_K }^2 \leq  ||\nu^K||_{\mathcal{A}_K }^2 - \frac{1}{2} ||\mathcal{A}_K\nu^K||_{\mathcal{D}_K^{-1}}^2 \quad  \forall \nu^K \in \mathfrak{M}^{K}.
\end{equation*}
\end{lemma}
\begin{proof}
According to  \eqref{eq: system1} and \eqref{coeeq: system1}, we know that the matrix $\mathcal{A}_K$ is a weakly diagonal dominant, since
\begin{equation*}
a_{0}+2\sum_{m=1}^{r}a_{m}+2\sum_{m=0}^{r-1}a_{m+\frac{1}{2}}=0.
\end{equation*}
Taking  $\mathcal{A}_K=\left\{a_{i,j}^{(K)}\right\}_{i,j=1}^{\infty}$ with $a_{i,j}^{(K)}=a_{|j-i|}^{(K)}$, it yields
$$ r_i^{(K)}:= \sum\limits_{j\neq i} |a_{i,j}^{(K)}| \leq  a_{i,i}^{(K)}=a_0.$$
Using the Gerschgorin circle theorem \cite[p.\,388]{Horn:13}, the eigenvalues of  $\mathcal{A}_K$ are in the disks centered at $a_{i,i}^{(K)}$ with radius
$r_i^{(K)}$. Namely,  the eigenvalues $\lambda$ of the matrix  $\mathcal{A}_K$ satisfy
$$  |\lambda -a_{i,i}^{(K)} | \leq r_i^{(K)}, $$
which leads to
$\lambda_{\max}(\mathcal{A}_K) \leq a_{i,i}^{(K)}+r_i^{(K)}< 2a_{i,i}^{(K)}=2a_{1,1}^{(K)}=2a_0 $.

From   Rayleigh theorem  \cite[p.\,235]{Horn:13}, we deduce
$$\lambda_{\max}(\mathcal{A}_{K})=\max_{x\neq 0}\frac{x^T\mathcal{A}_{K}x}{x^Tx}\quad\forall x \in  \mathbb{{R}}^n.$$
Take  $x=[1,0,\ldots,0]^T$. Then
$$\lambda_{\max}(\mathcal{A}_{K})\geq \frac{x^T\mathcal{A}_{K}x}{x^Tx}=a_{1,1}^{(K)}=a_0,$$
and
$$\lambda_{\max}\left(\left(\mathcal{D}_{K}\right)^{-1}\mathcal{A}_{K}\right)=\frac{\lambda_{\max}(\mathcal{A}_{K})}{a_{1,1}^{(K)}}= \frac{\lambda_{\max}(\mathcal{A}_{K})}{a_0}\in [1,2],$$
where  $\mathcal{D}_{K}$ is the diagonal of $\mathcal{A}_{K}$.

From Lemma \ref{xu96}, the Damped Jacobi method converges with $0<\omega<1$.
By the similar proof process  of Lemma \ref{cdlemma4.2}, we can check that $2\mathcal{D}_K-\mathcal{A}_K$ is symmetric positive definite.
Then using Lemma \ref{xu96} again, the Jacobi method converges.
Hence,  the damped Jacobi iteration
with relaxation parameter $0 < \omega \leq 1 $ converges.
The proof is completed.
\end{proof}

\begin{lemma}\label{adlm4.9}
Let $\mathcal{A}_K=\mathcal{A}_{h}$ be defined by \eqref{eq: system1}. Then
\begin{equation*}
   \min_{\nu^{K-1} \in \mathfrak{M}^{K-1} }||\nu^{K}-I_{K-1}^K\nu^{K-1}||_{\mathcal{D}_K}^2\leq \frac{1}{24} ||\nu^K||_{\mathcal{A}_K}^2 \quad  \forall \nu^K \in \mathfrak{M}^{K}.
\end{equation*}
\end{lemma}
\begin{proof}
 Let discrete  block-structured Laplacian-like operators be defined by
\begin{equation}\label{{ccs4.5}}
\mathcal{L}_j = \left[ \begin{array}{cc}  {L}^{N-1}_j  & ~ \mathbf{0}   \\
                                                           \mathbf{0}   &~  L^N_j  \end{array}\right]
                   = \left[ \begin{array}{cc}  {L}^{N-1}_j  &~  \mathbf{0}   \\
                                                           \mathbf{0}   &  ~ \mathbf{0}   \end{array}\right]
                   + \left[ \begin{array}{cc}  \mathbf{0}  &  ~\mathbf{0}   \\
                                                           \mathbf{0}   &  ~L^N_j  \end{array}\right],   ~~j=1,2,\dots,r,
\end{equation}
where the discrete Laplacian-like operators $\{L^{N-1}_j\}^{r}_{j=1}\in\mathds{R}^{(N-1)\times (N-1)}$, $\{L^N_j\}^{r}_{j=1}\in\mathds{R}^{N\times N}$ are  given in  Lemma  \ref{lem4.7}.

From \eqref{eq: system1}, the block-structured dense matrix $\mathcal{A}_K$ can be denoted  as the series sum with  Laplacian-like operators $\mathcal{L}_j $, i.e.,
\begin{equation}\label{ccs4.6}
\mathcal{A}_K = \left[\begin{array}{cc} A & B\\ B^T & \widehat{A} \end{array}\right]
= -\sum^r_{j=1}a_j\mathcal{L}_j  + \mathcal{B}
~~{\rm and}~~\mathcal{B} = \left(a_0 + 2\sum^r_{j=1}a_j \right)I + \left[\begin{array}{cc} \mathbf{0} & ~B\\ B^T & ~\mathbf{0} \end{array}\right]
\end{equation}
with $I$ an identity matrix.

Using  \eqref{coeeq: system1}, it yields
$
a_{0}+2\sum_{m=1}^{r}a_{m}+2\sum_{m=0}^{r-1}a_{m+\frac{1}{2}}=0,
$
 which implies that  $ \mathcal{B}$ is a weakly diagonally dominant symmetric matrix with positive entries on the diagonal and nonpositive off-diagonal entries.
Then $ \mathcal{B}$ is symmetric and semi-positive definite by applying the Gerschgorin circle theorem.

From  \eqref{ccs4.6}, \eqref{coeeq: system1} and Lemma \ref{lem4.7}, we obtain
\begin{equation*}
\begin{split}
\left|\left| \nu^K \right|\right|_{\mathcal{A}_K}& =\left( \mathcal{A}_K\nu^K, \nu^K\right)\geq -\frac{a_1}{2}\left( \sum^r_{j=1}\mathcal{L}_j\nu^K,\nu^K \right)
  \geq -\frac{ra_1}{4}\left( \mathcal{L}_1\nu^K,\nu^K \right) \\
 & \geq -\frac{ra_1}{8}\left( L^{2N-1}_1\nu^K,\nu^K \right) \geq \frac{a_0}{48}\left(L^{2N-1}_1\nu^K,\nu^K \right)
   \geq \frac{a_0}{24} \left|\left| \nu^K - I^K_{K-1}\nu^K \right|\right|^2\\
&   = \frac{1}{24} \left|\left| \nu^K - I^K_{K-1}\nu^J \right|\right|^2_{\mathcal{D}_K}.
\end{split}
\end{equation*}
The proof is completed.
\end{proof}

\begin{theorem}\label{the1}
Let $\mathcal{A}_K=\mathcal{A}_{h}$ be defined by \eqref{eq: system1}.
Then    the convergence factor of the TGM satisfies
\begin{equation*}
 \left|\left|S_KT_K\right|\right|_{\mathcal{A}_K}\leq\sqrt{47/48}<1.
\end{equation*}
\end{theorem}
\begin{proof}
According to Lemma \ref{lemma4.7}, \ref{lemma4.8},  \ref{lem 4.6 smooth} and \ref{adlm4.9},
the desired result is obtained.
\end{proof}

\section{Numerical results}

We employ the V-cycle block-structured AMG  described in Algorithm  \ref{Algorithm2:structured V-cycle} to solve the time-dependent  nonlocal problems in Section 2.
The stopping criterion is taken as
$$\frac{||r^{(i)}||}{||r^{0}||}<10^{-15},$$
 where $r^{(i)}$ is the residual vector after $i$ iterations, and the number of iterations $(m_1,m_2)=(1,1)$ and $(w_{pre},w_{post})=(1,1/2)$.
 In all tables, $N$ denotes the number of spatial grid points, and the numerical errors are measured by the $l_{\infty}$ (maximum) norm, ''rate'' denotes the convergence orders, ''CPU'' denotes the total CPU time in seconds (s) for solving the discretized systems, ''Iter'' denotes the average number of iterations required to solve algebraic systems  at each time level.

All numerical experiments are programmed in Matlab, and the computations are carried out  on a desktop with the configuration:
Intel(R) Core(TM) i7-7700 3.60 GHZ and 8 GB RAM and a 64 bit Windows 10 operating system.

First we consider the time-dependent nonlocal/peridynamic models    \eqref{modelnonlocal} and \eqref{ccs2.4} in Section 2.
The initial value and the forcing term are chosen such that the exact solution of the considered equations is
\begin{equation*}
\label{periodic}
u(x,t)=e^t\left(1+x\right)^6,\quad 0\leqslant x \leqslant 1,\  0\leqslant t \leqslant 1.
\end{equation*}
We apply the following  BDF4  method to  such  nonlocal models
\begin{equation}\label{bdf4}
\begin{split}
{\rm BDF4~scheme}: &\left(\frac{25}{12}I - \tau \mathcal{A}\right)U^k = 4U^{k-1} - 3U^{k-2} + \frac{4}{3}U^{k-3} - \frac{1}{4}U^{k-4} + \tau F^{k}.
\end{split}
\end{equation}
Here the operator $\mathcal{A}$ denotes the block-structured systems \eqref{3system}, \eqref{eq:nonsym system} and \eqref{eq: system1}, respectively.
It should be noted that the stability  and  convergence  analysis of time-dependent nonlocal problem \eqref{modelnonlocal} and \eqref{ccs2.4} can be seen in \cite{ACYZ:21,CQS:20,Chen:1--30,CCNWL:20}.
The convergence rate of the two-grid method for time-dependent  block-structured systems \eqref{eq: system1} can be directly obtained by Theorem \ref{the1} and \cite{CWCD:14}: here we omit the related derivation.

\subsection{Application in  nonlocal  model }
Table \ref{TT03b}  shows that  the BDF4 scheme \eqref{bdf4} for  time-dependent nonlocal model \eqref{modelnonlocal} has  the global {truncation error } $\mathcal {O}\left(\tau^4+h^{4-\gamma}\right)$
 and the computation cost is almost    $\mathcal{O}(N \mbox{log} N)$ operations.
\begin{table}[!th]
\vspace{-2mm}
\renewcommand{\captionfont}{\footnotesize}
\centering
\caption{Using  Galerkin approach    $\mathcal{A}_{k-1} = I_k^{k-1}\mathcal{A}_k I^k_{k-1}$ in \eqref{eq:block Galerkin} computed by Lemma \ref{lm3.1}  to solve the
 time-dependent nonlocal model \eqref{modelnonlocal} with $\tau=h=1/N$. }
\resizebox{\textwidth}{!}{
\begin{tabular}{ccccc|cccc|cccc}\hline
\multirow{2}*{$N$}&  \multicolumn{4}{c}{$\gamma=0$}   &\multicolumn{4}{c}{$\gamma=0.5$}        &\multicolumn{4}{c}{$\gamma=0.9$} \\
                          \cline{2-13}\noalign{\smallskip}
                 &Error     &Rate      &CPU     &Iter     &Error     &Rate    &CPU     &Iter       &Error           &Rate     &CPU      &Iter \\\hline
  $2^5$      &1.4460e-05    &~~~      &0.11s    &3      &1.9587e-05  &~~~~    &0.1167s   &3           &3.4809e-05  &~~~     &0.14s    &4 \\
  $2^6$      &9.7263e-07    &3.89       &0.39s    &3      &1.5160e-06  &3.6916    &0.4236s   &3         &3.5907e-06  &3.28     &0.47s    &4 \\
  $2^7$      &6.2993e-08    &3.95       &0.77s    &2      &1.1678e-07  &3.6984    &0.9123s   &3       &3.7667e-07  &3.25     &1.10s    &3  \\
  $2^8$      &3.9959e-09    &3.98       &2.26s    &2      &9.1285e-09  &3.6772    &2.1589s   &2        &4.0423e-08  &3.22     &2.94s    &3    \\\hline
\end{tabular}
}
\label{TT03b}
\end{table}

\subsection{Application in Peridynamic  model}
We further extend the V-cycle block-structured AMG   Algorithm  \ref{Algorithm2:structured V-cycle} to simulate the peridynamic model with
nonsymmetric indefinite   block-structured dense  systems and symmetric positive definite   block-structured dense  systems, respectively.
\subsubsection{ Nonsymmetric indefinite   block-structured dense  systems}
Table \ref{TT02}  shows that  the BDF4 scheme \eqref{bdf4} for  time-dependent peridynamic  model \eqref{modelnonlocal} with nonsymmetric indefinite   block-structured dense  systems has  the global {truncation error } $\mathcal {O}\left(\tau^4+h^{\max\left\{2,4-2\beta\right\}}\right)$  with $\delta=\mathcal{O}\left(h^\beta\right)$, $\beta\geq0$  and a computational cost of $\mathcal{O}(N \mbox{log} N)$ arithmetic operations.
\begin{table}[!th]
\vspace{-2mm}
\renewcommand{\captionfont}{\footnotesize}
\centering
\caption{ Nonsymmetric indefinite   block-structured dense  systems \eqref{eq:nonsym system}: Using  Galerkin approach    $\mathcal{A}_{k-1} = I_k^{k-1}\mathcal{A}_k I^k_{k-1}$ in \eqref{eq:block Galerkin} computed by Lemma \ref{lm3.1}  to solve the time-dependent peridynamic  model \eqref{modelnonlocal} with and $\tau=h=1/N$. }
\resizebox{\textwidth}{!}{
\begin{tabular}{cccccc|cccc }\hline\noalign{\smallskip}
           &\multirow{2}*{$N$}     &  \multicolumn{4}{c}{$\delta=1/4$}      &\multicolumn{4}{c}{$\delta=\sqrt{h}$}     \\
                          \cline{3-10}\noalign{\smallskip}
                                 &                 &Error             &Rate          &CPU         &Iter       &Error           &Rate       &CPU      &Iter     \\\hline
   \multirow{5}*{BDF4} & $2^5$      &4.3254e-05    &~~~~       &0.1460s    &7          &9.9042e-05  &~~~~    &0.2276s   &10              \\
                                 & $2^6$      &2.7166e-06    &3.9930       &0.3628s    &6          &9.3791e-06  &3.4005    &0.4737s   &8               \\
                                 & $2^7$      &1.7022e-07    &3.9963       &0.6684s    &5          &1.1886e-06  &2.9802    &1.5029s   &8                \\
                                 & $2^8$      &1.0652e-08    &3.9981       &2.0338s    &4          &1.3694e-07  &3.1177    &3.5816s   &7                 \\ \hline
%

\end{tabular}
}
\label{TT02}
\end{table}

\subsubsection{Symmetric positive definite   block-structured dense  systems}
Table \ref{TT01}  shows that  the BDF4 scheme \eqref{bdf4} for  time-dependent peridynamic  model \eqref{modelnonlocal} with symmetric    block-structured dense  systems has  the global truncation error  $\mathcal {O}\left(\tau^4+h^{\max\left\{2,4-2\beta\right\}}\right)$  with $\delta=\mathcal{O}\left(h^\beta\right)$, $\beta\geq0$  and a computational cost of $\mathcal{O}(N \mbox{log} N)$ arithmetic operations.

\begin{table}[!th]
\vspace{-2mm}
\renewcommand{\captionfont}{\footnotesize}
\centering
\caption{Symmetric positive definite   block-structured dense  systems \eqref{eq: system1}: Using  Galerkin approach    $\mathcal{A}_{k-1} = I_k^{k-1}\mathcal{A}_k I^k_{k-1}$ in \eqref{eq:block Galerkin} computed by Lemma \ref{lm3.1}  to solve the time-dependent peridynamic  model \eqref{modelnonlocal} with and $\tau=h=1/N$.}
\resizebox{\textwidth}{!}{
\begin{tabular}{cccccc|cccc }\hline\noalign{\smallskip}
           &\multirow{2}*{$N$}      & \multicolumn{4}{c}{$\delta=1/4$}      &\multicolumn{4}{c}{$\delta=\sqrt{h}$}     \\
                          \cline{3-10}\noalign{\smallskip}
                               &               &Error             &Rate          &CPU         &Iter        &Error           &Rate       &CPU      &Iter          \\\hline
 \multirow{5}*{BDF4} & $2^5$      &1.1628e-05    &~~~~       &0.2341s    &9          &2.5257e-05  &~~~~    &0.2987s   &12            \\
                              & $2^6$      &7.3840e-07    &3.9771       &0.4819s    &7          &2.3810e-06  &3.4070    &0.5999s   &10                \\
                              & $2^7$      &4.6514e-08    &3.9887       &1.0638s    &6          &2.9930e-07  &2.9919    &2.0004s   &10                 \\
                              & $2^8$      &2.9182e-09    &3.9945       &2.6782s    &5          &3.4366e-08  &3.1226    &4.6659s   &9                 \\ \hline
%
\end{tabular}
}
\label{TT01}
\end{table}

\section{Conclusions}
In this paper, we considered the solutions of block-structured dense and Toeplitz-like-plus-Cross    systems  arising from  nonlocal diffusion  problem and peridynamic problem.
We designed a AMG for block-structured dense and Toeplitz-like-plus-Cross systems, by making also use of fast Fourier transforms,
and we provided an estimate of the TGM convergence rate for the peridynamic  problem with symmetric positive definite block-structured dense linear  systems.
In this specific context, we answered the question on how to define coarsening and interpolation operators, when the stiffness matrix leads to nonsymmetric  systems \cite{Bre2010aggregation,MaNs:19}.
The simple (traditional) restriction operator and prolongation operator are employed for such  Toeplitz-like-plus-Cross systems, so that the entries of the sequence of subsystems are explicitly determined on different levels.

For the future, at least two questions arise and we plan to investigate them: more precisely we would like to consider the study of the TGM convergence analysis for nonsymmetric  block-structured dense  systems and the analysis of the full MGM for symmetric block-structured dense systems, based on the ideas presented in \cite{CES:20,ChenDeng2017}.


\begin{thebibliography}{99}


\bibitem{ACYZ:21}
{\sc G. Akrivis, M. H. Chen, F. Yu, and Z. Zhou},
{\em The energy technique for the six-step BDF method},
SIAM J. Numer. Anal.,  \textbf{59} (2021), pp. 2449--2472.

\bibitem{Andreu:10} {\sc  F. Andreu-Vaillo,   J. M. Maz\'{o}n, J. D. Rossi, and J. J.  Toledo-Melero},
 {\em Nonlocal Diffusion Problems},   Math. Surveys Monogr. 165,  AMS, Providence, RI, 2010.


\bibitem{Arico:07} {\sc  A. Aric\`{o} and  M. Donatelli},
{\em A V-cycle multigrid for multilevel matrix algebras: proof of optimality}, Numer. Math.,  \textbf{105}  (2007), pp. 511--547.

\bibitem{Arico:04} {\sc  A. Aric\`{o}, M. Donatelli, and S. Serra-Capizzano},
{\em V-cycle optimal convergence for certain (multilevel) structured linear systems}, SIAM J. Matrix Anal. Appl.,  \textbf{26}  (2004), pp. 186--214.

\bibitem{Aikinson:09} {\sc K. E. Atkinson},  {\em The Numerical Solution of Integral Equations of the Second Kind},
Cambridge University Press, New York, 2009.

\bibitem{Bates:06}
{\sc P.  Bates},
{\em On some nonlocal evolution equations arising in materials science}
    In: { H. Brunner, X. Zhao and  X. Zou} (eds.)
{ Nonlinear Dynamics and Evolution Equations}.
 in  Fields Inst. Commun. AMS, Providence, RI, (2006), pp. 13--52.

\bibitem{BenziGolub2005}
{\sc M. Benzi, G. H. Golub, and J. Liesen},
{\em  Numerical solution of saddle point problems},
Acta Numer., \textbf{14} (2005), pp. 1--137.


\bibitem{Bolten:15} {\sc M.  Bolten,    M.  Donatelli,   T.  Huckle,  and C.  Kravvaritis},
{\em Generalized grid transfer operators for multigrid methods applied on Toeplitz matrices},  BIT., {\bf55} (2015), pp.  341--366.


\bibitem{Bottcher:05}
{\sc  A. B\"{o}ttcher and S. M. Grudsky},
{\em Spectral Properties of Banded Toeplitz Matrices},  SIAM, Philadelphia, PA, 2005.


\bibitem{Bre2010aggregation}
{\sc M. Brezina, T. Manteuffel, S. Mccormick, J. Ruge, and G. Sanders},
{\em Towards adaptive smoothed aggregation ($\alpha$SA) for nonsymmetric problems},
SIAM J. Sci. Comput., \textbf{32} (2010), pp. 14--39.



\bibitem{CCNWL:20}
{\sc R. J. Cao, M. H. Chen, M. K. Ng, and Y. J. Wu},
 {\em Fast and high-order accuracy numerical methods for  time-dependent nonlocal problems in $\mathbb{R}^2$},
 J. Sci. Comput., \textbf{84} (2020), Paper No.8.




\bibitem{Chan:07} {\sc R. H.  Chan and  X. Q. Jin}, {\em An Introduction to Interative Toeplitz Solvers}, SIAM,  Phildelphia,  2007.

\bibitem{ChanNg:96} {\sc R. H.  Chan and  M. K. Ng}, {\em Conjugate gradient methods for Toeplitz systems}, SIAM Rev. \textbf{38} (1996),  pp. 427--482.


\bibitem{CCS:98}{\sc R. H. Chan, Q. S. Chang, and H. W. Sun},
{\em Multigrid method for ill-conditioned symmetric Toeplitz systems}, SIAM J. Sci. Comput.,  \textbf{19} (1998), pp. 516--529.


\bibitem{CES:20} {\sc  M. H. Chen, S. E. Ekstr\"{o}m, S. Serra-Capizzano},
{\em A Multigrid method for nonlocal problems: non-diagonally dominant or Toeplitz-plus-tridiagonal systems}, SIAM J. Matrix Anal. Appl.,  \textbf{41}  (2020), pp. 1546--1570.



\bibitem{Chen:1418--1438}
 {\sc M. H. Chen and  W. H. Deng},
{\em Fourth order accurate scheme for the space fractional diffusion equations},
 SIAM J. Numer. Anal., \textbf{52} (2014),  pp. 1418--1438.
 \bibitem{ChenDeng2017}
{\sc M. H. Chen, and W. H. Deng},
{\em Convergence analysis of a Multigrid  method for a nonlocal model},
SIAM J. Matrix Anal. Appl., \textbf{38} (2017), pp. 869--890.


\bibitem{ChenDeng2018V-cycle}
{\sc M. H. Chen, and W. H. Deng, and S. Serra-Capizzano},
{\em Uniform convergwnce of V-cycle Multigrid algorithms for two-dimensional fractional Feynman-Kac equation},
J. Sci. Comput., \textbf{74} (2018), pp. 1034--1059.






\bibitem{Chen:1--30}
{\sc M. H. Chen, W. Y. Qi, J. K. Shi and J. M. Wu},
 {\em A sharp error estimate of piecewise polynomial collocation for nonlocal problems with weakly singular kernels},
IMA J. Numer. Anal., \textbf{41} (2021), pp. 3145--3174.


\bibitem{CQS:20}{\sc  M. H. Chen, Y. F. Qi,  and J. K. Shi},
{\em Asymptotically compatible piecewise quadratic polynomial collocation for nonlocal model},
arXiv:2010.09215

\bibitem{ChenShi2021}
{\sc M. H. Chen, J. K. Shi, and X. B. Yin},
{\em Analysis of  (shifted) piecewise quadratic polynomial  collocation for nonlocal diffusion model}, Under Review

\bibitem{CWCD:14}
{\sc  M. H. Chen, Y. T. Wang, X. Cheng, and  W. H. Deng},
{\em Second-order LOD multigrid method for multidimensional Riesz fractional diffusion equation},
 BIT, \textbf{54} (2014), pp. 623--647.




\bibitem{DDGGTZ:20}{\sc  M. D'Elia, Q. Du, C. Glusa, M. Gunzburger, X. C. Tian, and Z. Zhou}, {\em Numerical methods for nonlocal and fractional models},
Acta Numer.,  \textbf{29} (2020), pp. 1--124.

\bibitem{WILEY2020BLOCKKt}
{\sc M. Donatelli, P. Ferrari, I. Furci, S. Serra-Capizzano, and  D. Sesana},
{\em  Multigrid methods for block-Toeplitz linear systems: convergence analysis andapplications},
Numer. Linear Algebra. Appl., \textbf{28} (2021), No. e2356

\bibitem{DMS:18}  {\sc M. Donatelli, M. Mazza, and  S.  Serra-Capizzano},
{\em Spectral analysis and multigrid methods for finite volume approximations of space-fractional diffusion equations}, SIAM J.  Sci. Comput.,
                         \textbf{40}  (2018), pp. A4007--A4039.

\bibitem{Du:19}
{\sc Q. Du},
{\em  Nonlocal Modeling, Analysis, and Computation},
 SIAM, Philadelphia, 2019.



 \bibitem{Du:12}
{\sc Q. Du,  M. Gunzburger,  R. Lehoucq, and K. Zhou},
{\em Analysis and approximation of nonlocal diffusion problems with volume constraints},
  SIAM Rev.,  \textbf{56} (2012), pp. 676--696.




\bibitem{Fiorentino:96}  {\sc G. Fiorentino and  S.  Serra},
{\em Multigrid methods for symmetric positive definite block Toeplitz matrices with nonnegative generating functions},
SIAM J.  Sci. Comput., \textbf{17}  (1996), pp. 1068--1081.



\bibitem{Horn:13}
 {\sc  R. A. Horn and   C. R. Johnson}, {\em Matrix Analysis},
Cambridge University Press, New York,  2013.

\bibitem{LTTF:21}{\sc Y. Leng, X. Tian, N. Trask, and J. T. Foster}, {\em Asymptotically compatible reproducing kernel collocation and meshfree integration for nonlocal diffusion},
SIAM J. Numer. Anal., \textbf{59} (2021), pp. 88--118.

\bibitem{MaNs:19}{\sc T. Manteuffel and S. Southworth}, {\em Convergence in norm of nonsymmetric algebraic multigrid},
SIAM J. Sci. Comput., \textbf{41} (2019), pp. S269-S296.


\bibitem{NgPanWeighted2014}
{\sc M. K. Ng, J. Y. Pan},
{\em  Weighted Toeplit regularized least squares computation for image restoration},
SIAM J. Sci.Comput.,  \textbf{36} (2014), pp. B94-B121.



\bibitem{NotayStokes2016}
{\sc Y. Notay},
{\em  A new algebraic multigrid approach for Stokes problems},
Numer. Math.,  \textbf{132} (2016), pp. 51-84.


\bibitem{PalSim2016BIT}
{\sc D. Palitta and V. Simoncini},
{\em  Matrix-equation-based strategies for convection-diffusion equations},
BIT Numer. Math., \textbf{56} (2016), pp. 751--776.

\bibitem{Pang:12} {\sc  H. Pang and H. Sun},
{\em  Multigrid method for fractional diffusion equations}, J. Comput. Phys.,  \textbf{231} (2012),  pp. 693--703.


\bibitem{QSW:21}
{\sc W. Y. Qi, P. Seshaiyer, J. P.  Wang,},
{\em  Finite element method with the total stress variable for Biot's consolidation model},
Numer. Methods Partial Differential Eq.,  \textbf{37} (2021), pp. 2409-2428.






\bibitem{RugeSIAM1987} {\sc J.  Ruge and K.   St\"{u}ben}, {\em Algebraic multigrid},
in {\em Multigrid Methods}, Ed: S. McCormick, SIAM, 1987,  pp. 73--130.


\bibitem{SaadInterativeSiam2003}
{\sc Y. Saad},
{\em  Iterative Methods for Sparse Linear Systems},
SIAM , Philadelphia, 2003.



\bibitem{Serra:02} {\sc S. Serra-Capizzano},
 {\em Convergence analysis of two-grid methods for elliptic Toeplitz and PDEs matrix-sequences}, Numer. Math.,  \textbf{92}  (2002), pp. 433--465.

 \bibitem{Silling:00}{\sc S. A. Silling},  {\em Reformulation of elasticity theory for discontinuities and long-range forces},
J. Mech. Phys. Solids, \textbf{48} (2000), pp. 175--209.


\bibitem{Simoncini2020IMA}
{\sc V. Simoncini},
{\em  On the numerical solution of a class of systems of linear matrix equations},
IMA J. Numer. Anal., \textbf{40} (2020), pp. 207-225.

\bibitem{Sivas2021siam}
{\sc A. A. Sivas, B. S. Southworth, and  S. Rhebergen},
{\em  Air algebraic multigrid for a space-time hybridizable discontinuous Garlerkin discretization of advection(-diffusion)},
SIAM J. Sci. Comput., \textbf{43} (2021), pp. A3393-A3416.

\bibitem{Stewart:94}  {\sc W. J. Stewart},
{\em Introduction to the Numerical Soluton of Markov Chains},
Princeton Univ. Press, 1994.


\bibitem{Tian:13} {\sc X. C.  Tian  and  Q. Du},
{\em Analysis and comparison of different approximations to nonlocal diffusion and linear peridynamic equations},
SIAM J.  Numer. Anal.,  \textbf{51} (2013), pp. 3458--3482.



\bibitem{VargaSpringer2000}
{\sc R. V. Varga},
{\em  Matrix Iterative Analysis},
Springer, Berlin Heidelberg, 2000.



\bibitem{WMPGW:21}  {\sc MT. A. Wiesner, M. Mayr, A. Popp, M. W. Gee, and W. A. Wall},
{\em Algebraic multigrid methods for saddle point systems arising from mortar contact formulations},
 Int. J. Numer. Methods Eng.,  \textbf{122}  (2021), pp. 3749--3779.


\bibitem{XuMMEPDES1996}
{\sc J. Xu},
{\em  An introduction to multilevel methods, in wavelets, Iterative Methods for Sparse Linear Systems, Multilevel Methods and Elliptic PDEs}, Leicester, 1996,
M.Ainsworth, J. Levesley, W. A. Light, and M. Marletta, eds., Oxford Universty Press, New York, 1997,  pp. 213-302.

\bibitem{Xu:17} {\sc   J.  Xu and L. Zikatanov}, {\em Algebraic multigrid methods}, Acta. Numerica.,  \textbf{26} (2017), pp. 591--721.





\end{thebibliography}
\end{document}